\documentclass[a4paper,11pt]{article}

      \usepackage[english]{babel}

      \usepackage[latin1]{inputenc}
      \usepackage{graphicx}
      \usepackage{multicol}
      \usepackage{graphics,epsfig}
      \usepackage{amsmath,amsfonts,amssymb,amsthm}
     \usepackage{dsfont}
\usepackage{geometry}
      \newtheorem{thm}{Theorem}[section]
      \newtheorem{propo}{Proposition}[section]
      \newtheorem{prop}{Properties}[section]
      \newtheorem{Def}{Definition}[section]
      \newtheorem{rmq}{Remark}[section]
      \newtheorem{lem}{Lemma}[section]
      \newtheorem{nota}{Notation}[section]
      
\title{\bf \large  Exponential speed of uniform  convergence of the cell density toward equilibrium for subcritical mass in a Patlak-Keller-Segel model }
 \author{\normalsize \bf Alexandre MONTARU\\
 \footnotesize Universit\'e Paris 13, Sorbonne Paris Cit\'e, \\
\footnotesize LAGA, CNRS, UMR 7539,\\ 
\footnotesize F-93430, Villetaneuse, France. \\
\small \textit{montaru@math.univ-paris13.fr} }

\date{}
\begin{document}

\maketitle  

\begin{abstract} \footnotesize
This paper is concerned with a chemotaxis aggregation model for cells, more precisely with a parabolic-elliptic semilinear Patlak-Keller-Segel system in a ball of $\mathbb{R}^N$ for $N\geq 2$. For $N=2$, this system is well known for its critical mass $8\pi$. It has been proved in \cite{Montaru2} that it also exhibits a critical mass phenomenon for $N\geq 3$. The main result of this paper is the exponential speed of uniform convergence of radial solutions toward the unique steady state in the subcritical case for $N\geq 2$.
We stress that this covers in particular the classical Keller-Segel system with $N=2$, and that the result improves on the known results even for this most studied problem.
 A key tool is an associated
one-dimensional degenerate parabolic problem $(PDE_m)$ where $m$ is proportional to the total mass of cells. The proof exploits its formal  gradient flow structure $u_t=-\nabla \mathcal{F}[u(t)]$ on an "infinite
dimensional Riemannian manifold". In particular, we  show a new Hardy type inequality, equivalent to the strict convexity of $\mathcal{F}$ at any steady state of subcritical mass, which heuristically explains  the exponential speed of convergence.
\end{abstract}

\tableofcontents

\section{Introduction}
\subsection{Origin of the problem}
In this paper, we are interested in the speed of convergence toward steady 
states  of solutions of the following problem, called 
$(PDE_m)$ :
\begin{eqnarray}
\label{equ_u_1}
\boxed{u_t=x^{2-\frac{2}{N}}\,u_{xx}+u\,{u_x}^q } & \qquad t>0 & 0<x\leq 1\\
\label{equ_u_2}
u(t,0)=0 &\qquad t\geq 0&\\
\label{equ_u_3}
u(t,1)=m &\qquad t\geq 0&\\
\label{equ_u_4}
u_x(t,x)\geq 0 &\qquad t> 0& 0\leq x\leq 1,
\end{eqnarray}
where  $m\geq 0$, $N$ is an integer greater or equal to $2$ and $q$
is the critical exponent, i.e. $$\boxed{q=\frac{2}{N}}\;.$$

Note that this parabolic problem has degenerate diffusion since $x^{2-\frac{2}{N}}$ vanishes at $x=0$ and that its nonlinearity involves
the gradient and is moreover non Lipschitz when $N\geq 3$ since $0<q<1$.\\

Problem $(PDE_m)$ arose for $N=2$ in the articles \cite{BKLN} of P. Biler, G. Karch, P. Laurençot and T. Nadzieja and  \cite{Souplet-Nikos} of N. Kavallaris and P. Souplet and then in  \cite{Montaru1,Montaru2} for $N\geq 3$ as a key tool 
    in the study of radial solutions of 
 the following  chemotaxis system $(PKS_q)$, supposed to describe a collection of 
cells diffusing in the open unit ball $D\subset \mathbb{R}^N$  and emitting a chemical which attracts themselves : 
\begin{eqnarray}
\label{PKS}
\rho_t= \Delta \rho-\nabla [\rho^q \, \nabla c]&\qquad t>0& \mbox{on } D\\
- \Delta c= \rho&\qquad t>0& \mbox{on } D,
\end{eqnarray}
with the following boundary conditions :  
\begin{equation}
\label{BC-rho}
\frac{\partial \rho}{\partial \nu}-\rho^q \, \frac{\partial c}{\partial \nu}=0 \mbox{ 
\quad  on } \partial D
\end{equation}
\begin{equation}
\label{BC-c}
c=0 \mbox{\quad  on }\partial D ,
\end{equation}
where $\rho$ is the cell density and $c$ the chemoattractant concentration. \\

Note that this  model  relies on the following assumptions : 
\begin{itemize}
\item Cells  diffuse much more slowly than the chemoattractant.
\item The cell flux $\vec{F}$ due to the chemoattractant is here described by
$\vec{F}=\chi \, \nabla c$ where   
$$\chi(\rho)=\rho^q$$
is the sensitivity of cells to the chemoattractant.
 \item On the boundary $\partial D$, there is a no flux condition  for 
$\rho$  
and a Dirichlet conditions for $c$.
\end{itemize}

This system $(PKS_q)$ is a particular case of the Patlak-Keller-Segel  model. To know  more about the latter, the reader can refer to the original works \cite{Patlak} of C.S. Patlak and \cite{KS} of E.F. Keller and L.A. Segel. For a review on mathematics of chemotaxis, see the chapter written by M.A. Herrero in \cite{Herrero} and the article \cite{HP} of  T. Hillen and K. J. Painter. For a review on the Patlak-Keller-Segel model, 
see both articles of D. Horstmann \cite{Horstmann1,Horstmann2}.\\
We also would like to very briefly recall some important results for the case  $N=2$ and $q=1$: \\
- It is known thanks to the works \cite{HV} of M.A. Herrero and J.L. Velazquez and  \cite{BKLN} of P. Biler, G. Karch, P. Laurençot and T. Nadzieja  that $8\pi$ is a critical mass for radial solutions in a ball. \\- In the case of the whole plane $\mathbb{R}^2$, this system has a similar behaviour. See  
\cite{Dolbeault-Perthame} by J. Dolbeault and B. Perthame, \cite{BCM} by A. Blanchet, J.A. Carrillo and N. Masmoudi, \cite{BKLN2} by P. Biler, G. Karch, P. Laurençot and T. Nadzieja and \cite{BDP} by A. Blanchet, J. Dolbeault and B. Perthame.\\
- For general solutions in a bounded domain of $\mathbb{R}^2$, the results are slightly different since  for a mass $4\pi$ blow-up at a point of the boundary of the domain can occur (see the book \cite{Suzuki} of T. Suzuki). 
\\

We now want to recall what is essential to 
know about the relation between problems
$(PKS_q)$ and  $(PDE_m)$ (much more
 can be found  
in \cite{Montaru2}):
\begin{itemize}
\item $m$ is proportional to the total mass of cells  
$\int_B \rho $ which is a conserved quantity in time.
\item The derivative of $u$ is the quantity with physical interest since $u_x$ is 
proportional to the cell density $\rho$, up to a rescaling in time and a change of 
variable.\\
 More precisely, denoting $\rho(t,y)=\tilde{\rho}(t,|y|)$ 
for $t\geq 0$ and $y\in\overline{D}$, we have 
$$ \tilde{\rho}(t,x)=N^{\frac{2}{q}}\,u_x(N^2\,t,x^N) \mbox{\qquad for all }x\in [0,1].$$
\item The power 
$q=\frac{2}{N}$ is critical. Indeed, as a particular case of \cite{HW} by D. Horstmann and M. Winkler, we know  that the solutions are global in time  when $q<\frac{2}{N}$ and can blow up if $q>\frac{2}{N}$
\end{itemize}
\bigskip
From now on, we will only focus on problem $(PDE_m)$, which becomes 
our chemotaxis model.
We will now list some facts that we have obtained in \cite{Montaru1,Montaru2}
for $N\geq 3$ and will later establish some similar results that we need for the case
 $N=2$.\\

\subsubsection{Case of dimension $N\geq 3$}
In \cite{Montaru1}, we have proved  the existence of a unique maximal classical  solution u 
of problem $(PDE_m)$ with initial condition $u_0\in Y_m$ and  existence time $T_{max}=T_{max}(u_0)>0$, where we denote
$$\boxed{Y_m=\{u\in C([0;1]),\:u \mbox{ nondecreasing },\; u'(0) \mbox{ exists, } u(0)=0,\;u(1)=m\}}$$ 
and "classical" means here that
$$u\in C([0,T_{max})\times [0,1])\bigcap
C^1((0,T_{max})\times [0,1])\bigcap C^{1,2}((0,T_{max})\times (0,1]).$$
Actually, we obtained more information about the regularity of the solutions and will refer to \cite{Montaru1} when necessary. \\

In \cite{Montaru2}, we showed that the stationary solutions of $(PDE_m)$ are the restrictions to $[0,1]$ 
of a family of functions $(U_a)_{a\geq 0}$ on $[0,+\infty)$ with the following simple structure :
\begin{itemize}
\item $U_1\in C^1([0,1])\cap C^2((0,1])$, $U_1(0)=0$, $\dot{U_1}(0)=1$, $U_1$ is increasing on $[0,A]$ for some $A>0$ and reaches its maximum $M$ at 
$x=A$ after which $U_1$ is flat.
\item All $(U_a)_{a\geq 0}$ are obtained by dilation of $U_1$, i.e. $U_a(x)=U_1(ax)$ for all $x\geq 0$.
\end{itemize}

An easy consequence of this description is that 
\begin{itemize}
\item If $0\leq m < M$, then there exists a unique stationary solution. The latter is given by $U_a|_{[0,1]}$, where $a=a(m)\in [0,A)$ is uniquely determined by $m$.
\item If $m=M$, there exists a continuum of steady states :  $\left(U_a|_{[0,1]}\right)_{a\geq A}$.\\ Note that the corresponding cell densities have their support strictly inside~$D$ when $a>A$.
\item If $m>M$, there is no stationary solution.
\end{itemize}

We call $M$ the critical mass of problem $(PDE_m)$, which is justified 
by the following result proved in \cite[Theorems 1.2 and 1.3]{Montaru2}, valid for any  $u_0\in Y_m$ :
\begin{itemize} 
\item If $m\leq M$, then $$T_{max}(u_0)=+\infty$$ and there exists $a\geq 0$ such that
 $$u(t) \underset{t \rightarrow + \infty}{ \longrightarrow } 
U_a \mbox{ \quad in \quad } C^1([0,1]).$$ 
More precisely, $a=a(m)\in [0,A)$   if $0\leq m<M$ and $a\geq A$ if $m=M$.
\item If $m>M$,  then  
$$T_{max}(u_0)<\infty.$$
\end{itemize}

\subsubsection{Case of dimension $N=2$}
For $N=2$, there is also such a critical mass phenomenon, well studied,
with critical mass $M=2$ corresponding to $8\pi$ in the original
Patlak-Keller-Segel model $(PKS_1)$ (see \cite{BKLN,HV}).\\

Problem $(PDE_m)$ then reads
\begin{eqnarray}
\label{n2_equ_u_1}
\boxed{u_t=x\,u_{xx}+u\,u_x } & \qquad t>0 & 0<x\leq 1\\
\label{n2_equ_u_2}
u(t,0)=0 &\qquad t\geq 0&\\
\label{n2_equ_u_3}
u(t,1)=m &\qquad t\geq 0&\\
\label{n2_equ_u_4}
u_x(t,x)\geq 0 &\qquad t> 0& 0\leq x\leq 1,
\end{eqnarray}
where  $m\geq 0$.\\

It is easy to see that its stationary solutions are all $$(\left.U_a\right|_{[0,1]})_{a\geq 0}$$ where 
$$U_a(x)=U_1(ax)$$ 
and 
$$U_1(x)=\frac{x}{1+\frac{x}{2}} $$
for all $x\in[0,1]$, $a\geq 0$.\\

The description of the set of steady states easily gives :

\begin{itemize}
\item
If $m<2$, there exists a unique classical steady state  of problem $(PDE_m)$, namely $\left.U_a\right|_{[0,1]}$ where 
$$a=a(m)=\frac{m}{1-\frac{m}{2}}\in [0,+\infty). $$
\item If $m\geq 2$, there is no classical stationary solution of problem 
$(PDE_m)$ but only a singular one $\overline{U}=m$ (singular in the sense that the boundary condition at $x=0$ is lost).
\end{itemize}

\begin{rmq}
A deep difference with the case $N\geq 3$ is that the steady states 
here do not reach their upper bound 2 and that the critical value switches from the regular
to the singular regime. \\
Actually, for all $a> 0$,  $\dot{U_a}>0$ on $[0,1]$. We will 
see in Theorem \ref{thm_existence_u} vii)
that this property is shared with the solution u at any 
time $t>0$, which means, coming back to the cell density
 interpretation, that cells are present in the whole ball $D$. 
This is in contrast with the case $N\geq 3$, where, 
at least in the critical mass case, the cells are sometimes 
present only in a ball strictly inside $D$.
\end{rmq}

It is  possible to show a similar result as 
for $N\geq3$, i.e. that if $0\leq m<2$, for any $u_0\in Y_m$, then
$$T_{max}(u_0)=+\infty$$ and  $$u(t) \underset{t \rightarrow + \infty}{ \longrightarrow } 
U_a \mbox{ \quad in \quad } C^1([0,1])$$ 
where $$a=a(m)=\frac{m}{1-\frac{m}{2}}\in [0,+\infty). $$

In \cite{BKLN}, for the subcritical case $0\leq m< 2$, the exponential speed of convergence of $u(t)$ toward the unique stationary solution $U_{a(m)}$ as $t\rightarrow +\infty$ 
was proved  for all $L^p$ norms with $1\leq p<\infty$ when the initial condition  $u_0$ is continuous and nondecreasing with $u_0(0)=0$ and $u_0(1)=m$ (a larger class than $Y_m$) and also in $L^\infty$ norm for some initial conditions for which global in time $W^{1,\infty}$  bound is known (the result then following by interpolation between $L^1$ and $W^{1,\infty}$).\\
As far as we know, the mere convergence in $C^1$ norm was unknown, and
a stronger result (the exponential convergence in $C^1$ norm) will actually
be  obtained below, by a very different technique from that in \cite{BKLN}. See section \ref{N=2} for more details.

\subsection{Main result}

The main goal of this paper is to study  the speed of convergence  of solutions of $(PDE_m)$ toward the unique stationary solution $U_a$ for the 
subcritical case $0<m<M$ ($m=0$ being obvious since $u=0$ because  $u_0\in Y_0=\{0\}$) 
 when 
\begin{equation}
\label{H1}
\boxed{N\geq 2}\,.
\end{equation} 
 
From now on, we fix 
\begin{equation}
\label{H2}
\boxed{0<m<M}
\end{equation} 
and
\begin{equation}
\label{H3}
\boxed{u_0\in Y_m}\;.
\end{equation} 
We denote $u$ the global solution of $(PDE_m)$ with initial condition $u_0$. 
We know  that
$$u(t) \underset{t \rightarrow + \infty}{ \longrightarrow } 
U_a \mbox{ \quad in \quad } C^1([0,1]),$$
where $U_a=U_{a(m)}$ is the unique stationary state of problem $(PDE_m)$. \\

Building on this qualitative information, we shall obtain a stronger quantitative
one, namely the exponential speed of convergence in $C^1([0,1])$.

\begin{thm}
\label{thm_convergence_expo_C1}
Assume (\ref{H1})(\ref{H2})(\ref{H3}). \\
Let $U_a=U_{a(m)}$ be the unique stationary solution of  $(PDE_m)$, i.e. problem (\ref{equ_u_1})-(\ref{equ_u_4}), and let
 $\lambda_1=\lambda_1(a)>1$ be the best constant of the Hardy type inequality in Proposition \ref{prop_existence_lambda1} below.\\ Let $\lambda\in(0,\lambda_1-1)$.\\
 Then there exists $C=C(u_0,\lambda)>0$ such that for all $t\geq 1$,
$$\|u(t)-U_a\|_{C^1([0,1])}\leq C \exp(-\lambda\dot{U_a}(1)^q\;t).$$
\end{thm}

\begin{rmq} We recall that the derivative of u is, up to a multiplicative constant and a change of variables, the  radial part of the cell density $\rho$ in the original Patlak-Keller-Segel model $(PKS_q)$.\\
Hence, this result is  equivalent to the exponential speed of uniform convergence of $\rho(t)$
toward $\rho_a$ where  $\rho_a$ is the cell density corresponding 
to~$U_a$.
\end{rmq}

The proof of Theorem \ref{thm_convergence_expo_C1} consists of two steps: 
\begin{itemize}
 \item We first establish exponential convergence in an appropriate weighted $L^2$ norm, 
by means of a linearization procedure and a suitable Hardy type inequality.
\item We then deduce exponential $C^1$ convergence by using a smoothing effect
after a suitable tranformation of the equation.
\end{itemize}

In the next subsection, we describe the first step of the proof.

\subsection{A Hardy type inequality and exponential convergence in a $L^2$ weighted space}

The following result, which is a Hardy type inequality, 
requires as a natural framework the  two
Hilbert spaces $L\supset Y_m$ and $H$, where 
$$\boxed{L=L^2\left((0,1),\frac{dx}{x^{2-q}}\right)}$$ is equipped 
with the norm  $$\|h\|_L=\sqrt{\int_0^1 \frac{h^2}{x^{2-q}}}$$
and   
$$\boxed{H=\left\{h\in L, \int_0^1 \dot{h}^2 <\infty,\,
 h(0)=h(1)=0\right\}}$$  with the norm  
$$\|h\|_H=\sqrt{\int_0^1 \frac{h^2}{x^{2-q}}+\int_0^1 \dot{h}^2}.$$ Note that,
 actually, $H=H^1_0\subset C^{\frac{1}{2}}([0,1])$ and the norms on $H$ and $H^1_0$ are equivalent (see Remark \ref{H=H10}).
 
\begin{rmq} 
 It is very natural to introduce $L$  from the viewpoint of 
the evolution equation $(PDE_m)$. 
 This will be justified in the following heuristics subsection.
\end{rmq}

\begin{propo}
\label{prop_existence_lambda1} Let $a\in(0,A)$.\\
There exists $\lambda_1=\lambda_1(a)>1$ such that for all $h\in H$,
\begin{equation}
\label{inegalite_lambda_1}
\int_0^1 \frac{\dot{h}^2}{\dot{U_a}^q} \geq \lambda_1 \int_0^1 \frac{h^2}{x^{2-q}}.
\end{equation}
Moreover, there exists $\phi_1\in H$ such that there is equality if and only if $h=c\,\phi_1$ for some $c\in \mathbb{R}$.
\end{propo}
As will be explained with much more details in  subsection
 \ref{heuristics},
 the evolution 
problem 
$(PDE_m)$ can formally be seen as a gradient flow equation
$$u_t=-\nabla \mathcal{F}[u(t)] $$ on some 
``infinite dimensional Riemannian manifold'' $(\mathcal{M},g)$
where 
$$\mathcal{M}=\{ u\in Y_m^1, \;\dot{u}>0
 \text{ on } [0,1]\}$$
is an open set of the affine space 
$$ Y_m^1=Y_m\cap C^1([0,1])$$
and the metric $g$ is defined 
by 
$$g_{u}(h,h)=\int_0^1 \frac{h^2}{x^{2-q}\dot{u}^q}$$
for all $u\in \mathcal{M}$ and 
$h\in T_u\mathcal{M}$, $T_u\mathcal{M}$ denoting  the tangent space to $\mathcal{M}$ at 
$u$.\\
 The previous
 result is actually equivalent to the strict convexity of 
the Lyapunov functional $\mathcal{F}$
at $U_a$, which makes us expect an exponential speed 
of convergence toward $U_a$, measured with the Riemannian distance $d_{\mathcal{M}}(U_a,\cdot)$ defined by 
the metric $g$ (which is equivalent to $\|\cdot\|_L$ near $U_a$). \\
Its proof  relies on the theory of compact self-adjoint operators 
on a separable Hilbert space and on a technique used in the article 
\cite{Beesack}  of P.R Beesack about extensions of Hardy's inequality.\\
We enjoy the opportunity to thank Philippe Souplet for suggesting this 
reading. \\

The following result shows rigorously the expected exponential speed of convergence in $L$ :
\begin{lem}
\label{lem_convergence_expo_L2}
Under the assumptions of Theorem \ref{thm_convergence_expo_C1},  there exists $C=C(u_0,\lambda)>0$ such that 
$$\|u(t)-U_a\|_L\leq C \exp(-\lambda\dot{U_a}(1)^q\;t)$$
for all $t\geq 0$.
\end{lem}

This result, though not the strongest, is the core of our paper.
 Its proof is inspired by both the gradient
 flow structure of problem $(PDE_m)$   and 
 the fact that $\mathcal{M}$ is  an open set of 
an affine space, which allows us to consider all the 
situation from the viewpoint of $U_a$. More precisely, 
if we define $$h(t)=u(t)-U_a$$ and consider
 $$\gamma(t)=g_{U_a}(h(t),h(t)),$$ we want to
 get a differential inequality 
on the latter. Since $h$ satisfies 
\begin{equation}
\label{intro_equation_h}
h_t=L_{U_a}h+F(x,h,\dot{h})
\end{equation} 
where $$L_{U_a}=x^{2-q} \dot{U_a}^q \; \frac{d}{dx}\left[\frac{\dot{h}}{\dot{U_a}^q}\right]+\dot{U_a}^q \;h$$ is the linearized operator at $u=U_a$ and $F$ is some remainder term, 
we will have two parts to deal with in the derivative 
of $\gamma$. The first term can be managed thanks 
to the Hardy type inequality in Proposition
 \ref{prop_existence_lambda1}. The second  
 imposes to sacrifice a bit of the first one, but 
without any serious  damage since there was anyhow 
no hope to reach the limit case $\lambda=\lambda_1$, 
at least by this way. 

\begin{rmq}\quad
\begin{itemize}
\item[i)] We can show that the degenerate parabolic equation (\ref{intro_equation_h}) satisfied by $h$ 
is regularizing in time from $L$ to $C^1([0,1])$, at least for large time (see Lemma \ref{lem_estimation_C1_L2}).
This will be enough to deduce the exponential speed of 
convergence toward steady states in $C^1([0,1])$, i.e. Theorem \ref{thm_convergence_expo_C1}, as an easy consequence of Lemma \ref{lem_convergence_expo_L2}.

\item[ii)]The constant $C$ we get is unbounded as $\lambda \longrightarrow \lambda_1$ so that we cannot get the same result with $\lambda=\lambda_1-1$.
\item[iii)] We think that the upper rate $\lambda_1 \dot{U_a}(1)^q$ is not optimal. We believe $\lambda_1$ is but not $\dot{U_a}(1)^q$ because it follows from the following rough inequality~: $\int_0^1\frac{h^2}{x^{2-q}}\geq \dot{U_a}(1)^q \int_0^1 \frac{h^2}{x^{2-q}\dot{U_a}^q}$ for any $h\in L$ 
because $U_a$ is concave.
\item[iv)] 
In dimension $N\geq 3$, an interesting question is to know whether  
 the exponential
 speed of convergence  degenerates or not
for $a=A$. Indeed, we can see that $\lambda_1(A)=1$ 
since if we set $w_A=\left.\frac{d}{da}\right|_{a=A}U_a$, 
we remark that $-\frac{d}{dx}\frac{\dot{w_A}}{\dot{U_A}^q}=
\frac{w_A}{x^{2-q}}$. Hence we can guess that
 $\lambda_1(a)\rightarrow 1$ as $a\rightarrow A$. But, since the center manifold seems to be made of the steady states $(U_a)_{a\geq A}$, it is not clear that the exponential speed of convergence should disappear. \\
It would then be very different for the critical mass for $N=2$ and $q=1$ since the speed of convergence degenerates and is no longer exponential. This has been done in \cite{Souplet-Nikos}. It was known that infinite
time blow-up of $u_x$ occurs. Of course, uniform convergence toward the constant 
singular steady state $\overline{U}=2$ cannot hold
in this case since $u(t,0)=0$. However, the authors proved that
$|u(t)-2|_1 \sim C \sqrt{t} e^{-\sqrt{2t}}$ as $t \to \infty$.
\end{itemize}
\end{rmq}

\subsection{Heuristics}
\label{heuristics}

Although the proof of Lemma \ref{lem_convergence_expo_L2} (cf. sections 3-4) can be read without any reference to the following heuristic arguments, we think that they shed some light on the underlying ideas and on the intuition that led to the rigorous proof. 
Indeed, the latter is inspired by a gradient flow approach, in the spirit of the seminal work of F. Otto \cite{Otto}, a strategy which has already been used successfully for the Patlak-Keller-Segel model. For instance, applying these ideas to  system $(PKS_1)$ in $\mathbb{R}^2$ for the subcritical mass case, A. Blanchet, V. Calvez and J.A. Carrillo recovered in \cite{BCC} the global in time existence of weak solutions  and  V. Calvez and J.A. Carrillo proved in \cite{Calvez-Carrillo}  the exponential speed of convergence of radial solutions  toward equilibrium, but measured with the Wasserstein distance $\mathcal{W}_2$. \\

First, we would like  to recall a basic fact about gradient flows
in a Euclidean space which provides a sufficient condition to have
a exponential speed of convergence to the stationary point.
We will give its rigorous proof, even though  it
 is very
simple, because it is the scheme for proofs 
in a more general infinite dimensional setting, as we will then see 
 on a
well-known instance in an infinite dimensional Hilbert space.   
Finally, 
we will see, without searching to be rigorous,  
that these ideas are inspiring in the case of 
problem $(PDE_m)$ which turns out to define a 
gradient flow on an "infinite dimensional Riemannian manifold".\\

For basic knowledge about strict 
Lyapunov functional and Lasalle's invariance principle,
we refer the reader to \cite[Chapter 9]{Haraux} or to 
\cite[Appendix G]{QS}. We also  recall some useful properties
 in subsection \ref{Lyapunov_rappels}.\\

We consider the following 
differential equation in the Euclidean space 
($\mathbb{R}^n,\langle \cdot,\cdot\rangle)$ 
having a gradient flow structure, i.e. 
$$\dot{x}(t)=- \nabla F(x(t))$$ 
with $F:\mathbb{R}^n\rightarrow \mathbb{R}$ smooth. 
\begin{lem} \label{lem_cv_exp} Let $x_0\in  \mathbb{R}^d$.\\
If the trajectory starting from $x_0$ is relatively compact in 
$\mathbb{R}^d$ (then global), 
if there exists a unique stationary 
point $x_\infty$ and if $F$ is strictly convex at $x=x_\infty$, i.e. 
F satisfies for some $\alpha_1>0$, 
$$d^2F(x_\infty)(\dot{x},\dot{x})\geq 
\alpha_1 |\dot{x}|^2 \text{ for all }\dot{x}\in \mathbb{R}^n,$$

then for any $\alpha\in(0,\alpha_1)$, there exists 
$C=C(x_0,\alpha)>0$ such that for all $t\geq 0$
$$|x(t)-x_\infty|\leq C \exp(-\alpha\;t).$$
\end{lem}

\begin{proof}[Proof of Lemma \ref{lem_cv_exp}]

First, we observe that F is a strict Lyapunov function since
$$ \frac{d}{dt}F(x(t))=-|\nabla F(x(t))|^2.$$
Since the trajectory $(x(t))_{t\geq 0}$ starting from $x_0$ is
 relatively compact, i.e. bounded in the context of an Euclidean 
space, then from  Lasalle's invariance principle, 
the $\omega$-limit set is made of stationary points. 
But since there is only one stationary point $x_\infty$, i.e.
verifying
$$\nabla F(x_\infty)=0,$$  
we can deduce the convergence of $x(t)$ toward $x_\infty$.

This implies in particular that
 $x_\infty$ is the minimum of $F$,
 so that moreover $$d^2F(x_\infty)\geq 0.$$ 
It is then not surprising that the strict convexity 
assumption on F will give  information about the speed 
of convergence of $x(t)$ toward $x_\infty$.
Indeed, if we  denote $$h(t)=x(t)-x_\infty$$ and $$\gamma(t)=|h(t)|^2,$$ 
we have
$$\dot{\gamma}(t)=-2 \langle \nabla F(x(t)),h(t) \rangle.$$
But $\nabla F(x_\infty)=0$, so $$\nabla F(x(t))=d(\nabla F)(x_\infty).h(t)+\epsilon(h(t))h(t)$$ where $\epsilon(h)\underset{h\rightarrow 0}{\longrightarrow}0$. Hence,
$$\dot{\gamma}(t)=-2d^2F(x_\infty).(h(t),h(t))+\epsilon(h(t))|h(t)|^2$$
 Now, let $\alpha<\alpha_1$.\\
  Since $h(t)\underset{h\rightarrow 0}{\longrightarrow}0$, there exists $t_0>0$ such that for all $t\geq t_0$, $\epsilon(h(t))\leq 2(\alpha_1-\alpha)$. Then, for all $t\geq t_0$,
 $$\dot{\gamma}(t)\leq -2 \alpha \;\gamma(t),$$
 which implies $$\gamma(t)\leq \gamma(t_0)\exp(-2\alpha\,t)$$ for all $t\geq t_0$ and finally we have  for all $t\geq 0$,
 $$\gamma(t)\leq C\exp(-2\alpha\,t)$$
 where $C=C(x_0,\alpha)$ because $\gamma$ is bounded. Whence the result.
\end{proof}

\medskip

As said before, this scheme can also be used in an infinite dimensional setting, like a Hilbert space. For example, let us consider the heat 
equation with Dirichlet condition on an bounded domain $\Omega$
$$ u_t=\Delta u.$$ 
This equation defines a continuous dynamical system on $L^2(\Omega)$ endowed with its standard scalar product $(\cdot\,,\cdot)$. and is moreover  regularizing so that,  for $t>0$, $u(t)\in H^1_0(\Omega)$. If we define $$F(u)=\int_\Omega \frac{|\nabla u|^2}{2},$$
then for $t>0$, $$u_t=-\nabla F(u(t))$$
 since $$(\nabla F(u),h)=dF(u).h=\int_\Omega \nabla u \nabla h=-\int_\Omega \Delta u\;h=(-u_t,h)$$ for all $h\in H^1_0(\Omega)$.\\
 It is easy to see that $F$ is a strict Lyapunov function and that 0  is the only stationary solution since the only harmonic function in $\Omega$ vanishing on the boundary is the zero function. \\
 Moreover, since F is quadratic, 
 $$d^2F(u).(h,h)=2 F(h)=\int_\Omega |\nabla h|^2\geq \lambda_1(\Omega) \|h\|_{L^2(\Omega)}^2$$ by Poincar\'e inequality, where $\lambda_1(\Omega)$ is the first eigenvalue of $-\Delta$ in $\Omega$ with Dirichlet condition. \\
 The same computation
 as above shows that for any $u_0\in L^2(\Omega)$, for any $\lambda<\lambda_1(\Omega)$, there exists $C=C(u_0,\lambda)>0$ such that  for all $t\geq 0$,
 $$\|u(t)\|_{L^2(\Omega)}\leq C\exp(-\lambda\,t).$$
 Note that, actually, the proof also works 
for $\lambda=\lambda_1$ in this particular instance 
 because $F$ is quadratic so that $u\mapsto\nabla F(u)$ is linear
hence there is no $o(h)$ to deal with.

\medskip

Another more general setting where this method can be applied is that of "infinite dimensional Riemannian manifolds". This idea has been deeply exploited in
the very nice paper \cite{Otto} concerning the porous medium equation.\\
It turns out that problem $(PDE_m)$ has this kind of gradient flow structure and
 we will try to take advantage of it. In what follows,
we will consider the case of dimension $N\geq 3$
 but all this discussion can be made for the case $N=2$.\\

If we denote the "infinite dimensional 
manifold" (actually an open set of the affine space $Y_m^1$)
$$\boxed{\mathcal{M}=\{ u\in Y^1_m, \;\dot{u}>0 \text{ on } [0,1]\}}$$ where 
we recall that $$Y_m^1=Y_m\cap C^1([0,1]) ,$$ 
we know that for $t>0$, $u(t)\in Y_m^1$ and then for t large enough, $$u(t)\in \mathcal{M}$$ since 
$u(t) \underset{t \rightarrow + \infty}{ \longrightarrow } 
U_a $ in $ C^1([0,1])$ and $\dot{U_a}>0$ on $[0,1].$\\

We can 
define the "Riemannian metric" g on $\mathcal{M}$ by
\begin{equation}
\label{formule_g_u}
\boxed{g_{u}(h,k)=\int_0^1 \frac{hk}{x^{2-q}\dot{u}^q}}
\end{equation}
for any $u\in \mathcal{M}$ and any $(h,k)\in T_u\mathcal{M}^2$, 
where actually, for any $u\in \mathcal{M}$ $$T_u\mathcal{M}=\mathcal{T}$$
with
$$\mathcal{T}=\{h\in C^1([0,1]),\; h(0)=h(1)=0 \} $$ 
since $\mathcal{M}$ is an open set of the affine space 
$Y_m^1$ which has $\mathcal{T}$ as direction (actually, 
$Y_m^1=m \,Id_{[0,1]}+\mathcal{T}$).\\

Now, we recall  the strict Lyapunov functional $\mathcal{F}$  used in \cite{Montaru2} to prove convergence toward steady states in the critical and subcritical mass cases :
$$\boxed{\mathcal{F}[u]=\int_0^1 \frac{\dot{u}^{2-q}}{(2-q)(1-q)}-\frac{u^2}{2x^{2-q}}}\;. $$
$\mathcal{F}$ can be guessed by the following equivalent formulation of (\ref{equ_u_1})
\begin{equation}
\label{PDE_m_autre_formulation}
u_t=x^{2-q}\dot{u}^q\left[\frac{d}{dx}\frac{\dot{u}^{1-q}}{1-q}+\frac{u}{x^{2-q}}\right].
\end{equation}
It is easy to see formally that
$$ u_t=-\nabla \mathcal{F}[u(t)],$$
which explains intuitively why $\mathcal{F}$ is a strict Lyapunov functional for $(PDE_m)$.\\
Indeed, for any $h\in T_u\mathcal{M}$, we have by definition
$$g_u(\nabla\mathcal{F}[u],h)= d\mathcal{F}(u).h $$  and moreover, by formal computation and integration by parts, we get  $$d\mathcal{F}(u).h=\int_0^1 \frac{\dot{u}^{1-q}\dot{h}}{1-q}-\frac{u\,h}{x^{2-q}}=-\int_0^1 \left[\frac{\ddot{u}}{\dot{u}^q}+\frac{u}{x^{2-q}}\right]h=-g_u(u_t,h). $$
Since we study the subcritical mass case, there exists a unique steady state $U_a$ so we have to compute the second derivative of $\mathcal{F}$ at this point.
Formally, we get 
$$\boxed{d^2\mathcal{F}[U_a].(h,h)= \int_0^1 \frac{\dot{h}^2}{\dot{U_a}^q}-\frac{h^2}{x^{2-q}}}\; .$$
As explained before, since $u(t) \underset{t \rightarrow + \infty}{ \longrightarrow } 
U_a $ in $C^1([0,1])$ , then $U_a$ is the minimum of $\mathcal{F}$
so that we can naturally expect that, for any $h\in T_{U_a}\mathcal{M}$, $$d^2F(U_a).(h,h)\geq 0.$$
If we can prove the stronger result that for some $\alpha_1>0$, we have  for all $h\in T_{U_a}\mathcal{M}$, 
 $$d^2\mathcal{F}[U_a].(h,h)\geq \alpha_1\; g_{U_a}(h,h)$$
or equivalently that for some $\lambda_1>1$, for all $h\in T_{U_a}\mathcal{M}$,
\begin{equation}
\label{intro_lambda1}
\int_0^1 \frac{\dot{h}^2}{\dot{U_a}^q}\geq \lambda_1 \int_0^1\frac{h^2}{x^{2-q}} ,
\end{equation}
then we can hope to prove that the  speed of convergence is  exponential  as before. 

\begin{rmq}
\label{rmq_philippe}
We thank Philippe Souplet for pointing out the following intuitive explanation of the fact that $\lambda_1>1$ in the present context. Indeed, for the subcritical case, the steady states of (\ref{PDE_m_autre_formulation}) form an increasing family $(U_a)_{a\in (0,A)}$ of solutions of 
$$ \frac{d}{dx}f(\dot{u})+V(x)u=0$$
where f is the increasing function on $[0,+\infty)$ defined for all $v\geq 0$ by $$f(v)=\frac{v^{1-q}}{1-q}$$  and
$$V(x)=\frac{1}{x^{2-q}}>0.$$
Hence, for any $a\in (0,A)$, $w_a=\frac{d}{da}U_a>0$ and $w_a$ 
formally satisfies
$$ \frac{d}{dx}[f'(\dot{U_a})\dot{w_a}]+V(x)w_a=0.$$
If $\phi_1>0$ is an eigenvector for the first eigenvalue $\lambda_1$, i.e.
 satisfies
$$ \frac{d}{dx}[f'(\dot{U_a})\dot{\phi_1}]+\lambda_1\; V(x)\phi_1=0,$$
then it is easy to see by integration by parts that
$$(\lambda_1-1)\int_0^1 V\,w_a\,\phi_1=[f'(\dot{U_a})w_a \dot{\phi_1}]_0^1 >0$$
by Hopf maximum principle on the boundary. Therefore, $\lambda_1>1$.

\end{rmq}

But here, there is an additional difficulty since we have a "Riemannian structure". Indeed, the metric $g$ here depends on the point $u$, so that if we set $$\gamma_0(t)=g_{u(t)}(u(t)-U_a,u(t)-U_a)$$ and differentiate it, there will be an extra term. This strategy is in some sense very natural since it takes into account the gradient flow structure. Nevertheless, because of this extra term, we  preferred 
to also take advantage of the fact that $\mathcal{M}$ is an open set of an  affine space  by considering  $$\gamma(t)=g_{U_a}(u(t)-U_a,u(t)-U_a),$$ i.e. we fixed the point $U_a$ and consider the difference $u(t)-U_a$ belonging to the tangent space $T_{U_a}\mathcal{M}$. Hence, this strategy of linearization somehow uses both the gradient flow structure via the good relation between $g$ and the flow, and the "affine structure" because we can  fix $U_a$ and consider the situation from its viewpoint.\\

Finally, we also  remark that if $U\in \mathcal{M}$ is near of $U_a$, then all measures by the metrics  $g_u(h,h)$ are comparable to $$\|h\|_L=\sqrt{\int_0^1 \frac{h^2}{x^{2-q}}}.$$ Hence, recalling that the Riemannian metric $d_\mathcal{M}$ on $\mathcal{M}$ between $U_a$ and $U$
is defined by
$$d_\mathcal{M}(U_a, U)^2=\underset{\{u\in C^1([0,1],\mathcal{M}),\;u(0)=U_a,\,u(1)=U\}}{\inf} \int_0^1 g_{u(t)}(u_t,u_t)\, dt,$$
it is clear that $d_\mathcal{M}(U_a,U)$ is equivalent to 
$\|U-U_a\|_L $ for $U$ near of $U_a$.
This consideration naturally leads us to introduce the  
Hilbert space $L\supset Y_m$, where 
$$\boxed{L=L^2\left((0,1),\frac{dx}{x^{2-q}}\right)}.$$
It is also very natural to make the proof of the Hardy type inequality (\ref{intro_lambda1}) in a larger space than $T_{U_a}\mathcal{M}$, namely for all $h\in H$, 
 where  $H$ is  the  Hilbert space 
$$\boxed{H=\left\{h\in L, \int_0^1 \dot{h}^2 <\infty,\, 
h(0)=h(1)=0\right\}}$$  equipped with the norm 
 $$\|h\|_H=\sqrt{\int_0^1 \frac{h^2}{x^{2-q}}+\int_0^1 \dot{h}^2}.$$

\medskip

\textbf{Outline of the rest of the paper.} 
In section~2, we state some preliminary results for dimension $N=2$
which will be proved in the appendix.\\
The next sections are devoted to proofs.
In section~3, we will get the strict convexity of $\mathcal{F}$ (or $\mathcal{G}$ if $N=2$)
at $U_a$ by showing  its equivalent form expressed in the Hardy 
type inequality of  Proposition~\ref{prop_existence_lambda1}. \\
In section~4, we show Lemma~\ref{lem_convergence_expo_L2} which
 establishes the exponential speed of convergence toward 
the steady state  in~$L$. 
\\In section~5, we prove that the 
degenerate parabolic equation satisfied by $h=u-U_a$ is 
regularizing for large time from L to $C^1([0,1])$, i.e. 
Lemma \ref{lem_estimation_C1_L2} which therefore easily 
implies Theorem \ref{thm_convergence_expo_C1}. In the appendix, we also  recall some basic facts about continuous dynamical systems and Lyapunov functionals. \\

\section{Preliminary results for dimension $N=2$}
\label{N=2}
In this section, we focus on the most studied case of dimension 2, well-known for its critical mass $8\pi$ if we come back to the 
original Keller-Segel system (\ref{PKS}). Our aim is to state the results that lead us to Lemma \ref{N=2_lem_convergence_C1}, i.e. to
the $C^1$ convergence of $u(t)$ toward the unique steady state $U_a$ 
that we mentioned in the introduction.

We would like to remark that problem $(PDE_m)$ 
 is simpler
for $N=2$ (see (\ref{n2_equ_u_1})) than for $N\geq 3$ (see (\ref{equ_u_1})) since its nonlinearity is then locally Lipschitz
(even bilinear) in $(u,u_x)$. 
Accordingly, the convergence results for $N=2$, as well as the required wellposedness and regularity properties, can be proved by similar ideas as in \cite{Montaru1,Montaru2} which treat the case
$N\geq 3$. We point out that some of the wellposedness issues for $N=2$ have been addressed in \cite{Souplet-Nikos,BKLN}, but that they do not provide all the necessary properties that we need. Therefore, and also for the sake of completeness, we chose to give all the  proofs in Appendix, trying to be reasonably self-contained.

\subsection{Local wellposedness and regularity for problem $(PDE_m)$}

We first give a wellposedness and regularity theorem 
which requires the introduction of the following  "norm" 
$\mathcal{N}$ and some notation.
\begin{Def} For any real function $u$ defined on $(0,1]$, we set
$$\mathcal{N}[u]=\underset{x\in(0,1]}{\sup}\frac{u(x)}{x}.$$
\end{Def}

\begin{nota} Let $m\geq 0$ and $\gamma>0$. 
\begin{itemize}
\item $Y_m=\{u\in C([0;1]) \mbox{ nondecreasing},\; u'(0) \mbox{ exists, } u(0)=0,u(1)=m\}$.
\item $Y_m^1=Y_m\cap C^1([0,1])$.
\item  $Y_m^{1,\gamma}=\{u\in Y_m\cap C^1([0,1]), \;\underset{x\in (0,1]}{\sup}\; \frac{|u'(x)-u'(0)|}{x^\gamma}<\infty\}$.
\end{itemize}
\end{nota}

\begin{thm}
\label{thm_existence_u}
Let  $K>0$ and $u_0\in Y_m $ with $\mathcal{N}[u_0]\leq K$.
\begin{itemize}
\item[i)] There exists $T_{max}=T_{max}(u_0)>0$ and a unique maximal classical solution of $(PDE_m)$ 
with initial condition $u_0$, i.e.
$$u\in C([0,T_{max})\times [0,1])\bigcap
C^1((0,T_{max})\times [0,1])\bigcap C^{1,2}((0,T_{max})\times (0,1])$$
verifying (\ref{n2_equ_u_1})(\ref{n2_equ_u_2})(\ref{n2_equ_u_3})(\ref{n2_equ_u_4}) and $u(0)=u_0$.\\
Moreover, $u$ satisfies the following condition :
\begin{equation}
\label{maj_u_C1}
 \underset{t\in(0,T]}{\sup} \sqrt{t}\; \|u(t)\|_{C^1([0,1])}<\infty \mbox{ for any } T\in(0,T_{max}).
 \end{equation}
\item [ii)] There exists $\tau=\tau(K)>0$ such that $T_{max}\geq \tau$.
\item[iii)] Blow up alternative : 
$T_{max}=+\infty$ \quad or \quad  $\underset{t\rightarrow T_{max}}
{\lim}\mathcal{N}[u(t)]=+\infty $
\item [iv)]
$u\in C^\infty ( (0,T_{max})\times (0,1])$.
\item[v)] If $u_0\in Y_m^{1,\gamma}$ with $\frac{1}{2}<\gamma \leq 1$ 
then $u\in C([0,T_{max}),C^1([0,1]))$.
\item [vi)] For all $t\in(0,T_{max})$, $u(t)\in Y_m^{1,1}$.
\item[vii)] $u_x(t,x)>0$  for all $(t,x)\in (0,T_{max})\times [0,1]$.
\end{itemize}
\end{thm}

At least the four first points were known explicitly or 
implicitly (see \cite{Souplet-Nikos}). Concerning point vii), to our knowledge,  it was  only proved that for all
$t\in(0,T_{max})$, $$u_x(t,0)>0.$$  Although  vii) is expected, its
proof is rather technical and moreover this fact
 will turn out to be essential
in the proof of Lemma \ref{G_Lyapunov_functional}.

\subsection{Subcritical case : Lyapunov functional and convergence in $C^1([0,1])$}
\label{N=2-convergence}

From now on, we only focus on the subcritical case 
$$ \boxed{m<2}$$
which corresponds to mass lower than $8\pi$ for the original Keller-Segel system (\ref{PKS}).\\

Then, the classical solutions of $(PDE_m)$ are globally defined. More precisely :
\begin{lem}
\label{traj_globales}
 Let $m<2$ and $u_0\in Y_m$. 
 Then $$T_{max}(u_0)=+\infty.$$ 
\end{lem}

The next lemma, stating in particular the relative compactness of the trajectory $\{u(t),\,t\geq 1\}$ in $Y_m^1$ for any initial condition $u_0\in Y_m^1$, will also be useful to check that 
$(T(t))_{t\geq 0}$ defined below is a continuous dynamical system on $Y_m^1$.

\begin{Def}
Let $u_0\in Y_m^1$ and $t\geq 0$.\\ 
We define $T(t)u_0=u(t)$ where  $u$ is the classical solution of problem $(PDE_m)$ with initial condition $u_0$.
\end{Def}

\begin{lem}
\label{lem_compacite}
 Let $m<2 $,  $t_0>0$ and $K>0$.\\
Then, $\{T(t)u_0, \;\mathcal{N}[u_0]\leq K,\; t\geq t_0\}$ is relatively 
compact in $Y_m^1$.
\end{lem}

\begin{lem}
\label{T_dynamical_system}
$(T(t))_{t\geq 0}$ is a continuous dynamical system on $Y_m^1$.
\end{lem}

We now introduce a functional which is an analogue of $\mathcal{F}$ in the case $q=1$.
\begin{Def} Let
 $\mathcal{M}=\{ u\in Y_m^1,\; u_x>0 \text{ on } [0,1]\}$.\\
We define for all $u\in \mathcal{M}$, 
$$\mathcal{G}[u]=\int_0^1 u_x[\ln u_x-1]-\frac{u^2}{2x}.$$
\end{Def}
Indeed, we have the following result.
\begin{lem}
\label{G_Lyapunov_functional}
$\mathcal{G}$ is a strict Lyapunov functional for $(T(t))_{t\geq 0}$.
\end{lem}

As a consequence, we finally get :
\begin{lem}
\label{N=2_lem_convergence_C1}
Let $0\leq m<2$ and $u_0\in Y_m$. Then\\

$$u(t) \underset{t \rightarrow + \infty}{ \longrightarrow } 
U_a \mbox{ \quad in \quad } C^1([0,1])$$
where  
$$a=\frac{m}{1-\frac{m}{2}}. $$
\end{lem}

\section{A Hardy type inequality}

The aim of this section is to prove Proposition~\ref{prop_existence_lambda1}. 
First, we will need to establish some intermediate lemmas.
For reader's convenience, we recall that
$$L=L^2\left((0,1),\frac{dx}{x^{2-q}}\right) \qquad \text{ and} \qquad  
H=\left\{h\in L, \int_0^1 \dot{h}^2 <\infty,\, h(0)=h(1)=0\right\}$$
and that $L$ and $H$ are equipped with the following norms
$$\|h\|_L^2=\int_0^1 \frac{h^2}{x^{2-q}}\qquad \text{ and} \qquad  
\|h\|_H^2=\int_0^1 \frac{h^2}{x^{2-q}}+\int_0^1 \dot{h}^2.$$

\begin{rmq}
\label{H=H10}
Actually, we can see that $$H=H_0^1$$ and that $\|\cdot \|_{H_0^1}$ and $\|\cdot\|_H$ are equivalent. \\Indeed, $H\subset H^1_0$ with continuous embedding is obvious and the reverse  is also true by the standard Hardy  inequality 
\begin{equation}
\label{Hardy_standard}
\int_0^1 \frac{h^2}{x^{2}}\leq 4\int_0^1  \dot{h}^2.
\end{equation}
valid for any $h\in H^1_0$.
Note also that L and H are  separable Hilbert spaces. 
\end{rmq}

We will need the following compactness result.
\begin{lem}
\label{lem_compacite_H-L}
The imbedding $H\subset L$ is compact.
\end{lem}

\begin{proof}
For any $\alpha\in (0,1]$, we denote $$C_0^{\alpha}= \{h\in C^{\alpha}([0,1]),\; h(0)=0\}$$
the Banach space equipped with the norm $$\|h\|_{C_0^\alpha}=\underset{\underset{x\neq y}{(x,y)\in[0,1]^2}}{\sup}\frac{|h(x)-h(y)|}{|x-y|^\alpha}.$$
It is clear that $H\subset C_0^{\frac{1}{2}}$ with continuous imbedding
since if $h\in H$, $$|h(x)-h(y)|=|\int_y^x \dot{h}|\leq \sqrt{|x-y|}\sqrt{\int_0^1 \dot{h}^2}.$$
Now, let $\gamma \in (\frac{1-q}{2},\frac{1}{2})$.  The imbedding $C_0^{\frac{1}{2}}\subset C_0^{\gamma} $
is compact and the imbedding $C_0^{\gamma} \subset L$ is continuous since for all $h\in C_0^{\gamma}$,
$$\|h\|_L^2=\int_0^1  \frac{h^2}{x^{2-q}}\leq \|h\|_{C_0^\gamma}^2 \int_0^1  \frac{1}{x^{2-q-2\gamma}}$$
with  $\int_0^1  \frac{1}{x^{2-q-2\gamma}}<\infty$ since $2-q-2\gamma<1$.
\end{proof}

\medskip

The following lemma, whose proof relies on a technique used in 
\cite{Beesack} to get extensions of Hardy's inequality,
 will be essential in the proof of Proposition~\ref{prop_existence_lambda1}.

\begin{lem}
\label{lem_fonctionnelle_positive}
Let  $0<a<A$. Then, for all $h\in H$ 
\begin{equation}
\label{inegalite_fonctionnelle}
\int_0^1 \frac{\dot{h}^2}{\dot{U_a}^q}-\frac{h^2}{x^{2-q}}\geq 0
\end{equation}
with equality if and only if $h=0$.
\end{lem}

Before giving the proof, we recall some useful properties of $U_a$. 
For all $a\geq 0$,
\begin{equation}
 \label{equation_U_a}
x^{2-\frac{2}{N}}\ddot{U_a}+U_a\, \dot{U_a}^\frac{2}{N}=0
\end{equation}
and
$$\dot{U_a}(0)=a.$$
This implies the concavity of $U_a$, so $$ \dot{U_a}(1)\leq \dot{U_a}\leq a \text{ on }[0,1].$$
Moreover, for all $x\in [0,1]$,
$$U_a(x)=U_1(ax) .$$
Since $U_1$ is increasing on $[0,A]$ (and  flat after $x=A$) 
for some $A>0$, then
$$\text{ for } 0<a<A,\;\dot{U_a}>0 \text{ on } [0,1].$$

\begin{proof}
We denote for 	all $x\in[0,1]$, 
$$w_a(x)=\frac{d}{da}U_a(x)=x\;\dot{U_1}(ax).$$
We see that $ w_a>0$  
on  $(0,1]$ since $0<a<A$. \\
Moreover, for all $x\in[0,1]$,  noting first that 
$$\dot{w_a}(x)=\dot{U_1}(ax)+ax \ddot{U_1}(ax),$$
we have 
\begin{align*}
\frac{\dot{w_a}(x)}{\dot{U_a}^q(x)}&=\frac{\dot{U_1}(ax)^{1-q}}{a^q}+a^{1-q}x \frac{\ddot{U_1}(ax)}{\dot{U_1}(ax)^q}\\
&=\frac{\dot{U_1}(ax)^{1-q}}{a^q}-\frac{U_1(ax)}{ax^{1-q}} 
\text{ \quad by } (\ref{equation_U_a})
\end{align*}
then, we obtain 
\begin{align*}
\frac{d}{dx}\left[\frac{\dot{w_a}}{\dot{U_a}^q}\right]&=
(1-q)a^{1-q}\frac{\ddot{U_1}(ax)}{\dot{U_1}(ax)^q}+(1-q)\frac{U_1(ax)}{ax^{2-q}}-\frac{\dot{U_1}(ax)}{x^{1-q}}\\
&=-(1-q)a^{1-q}\frac{{U_1}(ax)}{(ax)^{2-q}}+(1-q)\frac{U_1(ax)}{ax^{2-q}}-\frac{w_a(x)}{x^{2-q}}=-\frac{w_a(x)}{x^{2-q}}
\end{align*}
again by (\ref{equation_U_a}), so we see that $w_a$ satisfies 
\begin{equation}
\label{equation_w_a}
\frac{d}{dx}\left[\frac{\dot{w_a}}{\dot{U_a}^q}\right]+\frac{w_a}{x^{2-q}}=0.
\end{equation}
We note that this equation could  also be obtained  by  differentiating $(\ref{equation_U_a})$ with respect to $a$.\\

The proof of Lemma \ref{lem_fonctionnelle_positive} will be made by density.  Therefore, we first make the following computation  for any $h\in C_c^\infty((0,1))$.
\begin{align*}
\int_0^1 \frac{[\dot{h}-\dfrac{\dot{w_a}}{w_a}h]^2}{\dot{U_a}^q}&=
\int_0^1 \frac{\dot{h}^2}{\dot{U_a}^q}+\int_0^1 \left[\frac{\dot{w_a}}{w_a}\right]^2 \frac{h^2}{\dot{U_a}^q}-\int_0^1 2h\dot{h}\,\left[ \frac{\dot{w_a}}{w_a}\frac{1}{\dot{U_a}^q}\right]\\
&=\int_0^1 \frac{\dot{h}^2}{\dot{U_a}^q}+\int_0^1 \left[\frac{\dot{w_a}}{w_a}\right]^2 \frac{h^2}{\dot{U_a}^q}+\int_0^1 h^2 \frac{d}{dx}\left[\frac{\dot{w_a}}{w_a}\frac{1}{\dot{U_a}^q}\right]\\
&=\int_0^1 \frac{\dot{h}^2}{\dot{U_a}^q}+\int_0^1 h^2\,
\left( \left[\frac{\dot{w_a}}{w_a}\right]^2 \frac{1}{\dot{U_a}^q}+ 
\frac{d}{dx}\left[\frac{\dot{w_a}}{w_a}\frac{1}
{\dot{U_a}^q}\right]\right)
\end{align*}
where we used $h(0)=h(1)=0$ in the integration by parts. \\
 From (\ref{equation_w_a}), we deduce
 \begin{equation*}
 \frac{d}{dx}\left[\frac{1}{w_a}\frac{\dot{w_a}}
{\dot{U_a}^q}\right]+\frac{1}{\dot{U_a}^q}\left[\frac{\dot{w_a}}
{w_a}\right]^2=\frac{d}{dx}\left[\frac{\dot{w_a}}
{\dot{U_a}^q}\right]\frac{1}{w_a}=-\frac{1}{x^{2-q}}
 \end{equation*}
 which, coming back to the previous computation, implies
 $$\int_0^1 \frac{\dot{h}^2}{\dot{U_a}^q}-\frac{h^2}{x^{2-q}}=\int_0^1 \frac{[\dot{h}-\dfrac{\dot{w_a}}{w_a}h]^2}{\dot{U_a}^q}\geq 0 .$$
 Let $h\in H$. Since $H=H^1_0(0,1)$ with equivalent norms, then
 $C_c^\infty((0,1))$ is dense in $H$ so 
 there exists $h_n\in C_c^\infty((0,1))$ such that
$h_n\rightarrow h$ in $L$ and $\dot{h_n}\rightarrow \dot{h}$ in $L^2((0,1),dx)$ with convergence almost everywhere in $(0,1)$ and domination by two functions respectively in $L$ and $L^2(0,1)$.
Hence, by Lebesgue's dominated convergence theorem, the previous equation is also valid for h.\\

Since $w_a\left(\frac{h}{w_a}\right)'=\dot{h}-\dfrac{\dot{w_a}}{w_a}h$,
there is equality if and only if $h=c\,w_a$ for some $c\in \mathbb{R}$ almost everywhere on $(0,1)$,  but actually everywhere  on $[0,1]$ since $h$ and $w_a$ are both continuous.
 Now, we note that  $h(1)=0$ and $w_a(1)>0$ since $a<A$,  so $h=c\,w_a$ implies $c=0$, i.e $h=0$.
\end{proof}
\medskip
\begin{proof}[Proof of Proposition \ref{prop_existence_lambda1}.]
The following procedure is standard.\\
Considering the  symmetric bilinear form $\Lambda$ defined on H  
by
$$ \Lambda(h,k)=\int_0^1 \frac{\dot{h}\dot{k}}{\dot{U_a}^q} \text{\qquad for all }(h,k)\in H^2,$$
 it is easy to see that $\Lambda$ is continuous and coercive. Hence,   we can apply the Lax-Milgram theorem and prove that 
for any $\varphi \in H'$, there exists a unique $h\in H$ such that
$$\Lambda(h,\cdot)=\varphi. $$ 

Thanks to Lemma \ref{lem_fonctionnelle_positive}, any $f\in L$ defines $\varphi_f\in H'$ by 
$$\varphi_f(k)=\int_0^1 \frac{f\,k}{x^{2-q}} \quad \text{ for all } k\in H. $$

We then define $T :L\rightarrow  L$ by $Tf=h$ where $h\in H$ is such that 
$\Lambda(h,\cdot)=\varphi_f. $ It is easy to see that $T$ is 
self-adjoint, continuous (thanks to Lax-Milgram) and even compact, 
thanks to  Lemma \ref{lem_compacite_H-L}. \\

The end of the proof, which relies on the theory of compact self-adjoint 
operators on a separable Hilbert  space, is completely similar to that of \cite[Theorem 2, p.336]{Evans}. 
Moreover, since the infimum in 
$$\lambda_1=\underset{\underset{h\neq 0}{h\in H}}{\inf} \frac{\Lambda(h,h)}{\|h\|_L^2}$$
is reached, then 
 Lemma \ref{lem_fonctionnelle_positive} implies $\lambda_1>1$.

\end{proof}

\section{Convergence with exponential speed in $L^2\left((0,1),\frac{dx}{x^{2-q}}\right)$}

\begin{proof}[Proof of Lemma \ref{lem_convergence_expo_L2}.]
We  let
\begin{equation}
\label{formule_u_h}\nonumber
u=U_a+h.
\end{equation}
To get the result, it is equivalent to show the existence of $C>0$ such that 
$$\gamma(t)\leq C \exp(-2\,\lambda\,\dot{U_a}(1)^q\;t) $$
where 
\begin{equation}\nonumber
\label{def_gamma}
\gamma(t)=g_{U_a}(h(t),h(t))=\int_0^1 \frac{h(t)^2}{x^{2-q}\dot{U_a}^q}.
\end{equation}

An easy computation shows that for any $t>0$ and any $x\in(0,1]$, 

\begin{equation}
\label{equation_h}
h_t=L_{U_a}h+F(x,h,\dot{h})
\end{equation}
 where 
\begin{align}
\label{formule_L1}
L_{U_a}h
&=\left[x^{2-q}\right] \; \ddot{h}+\left[q\frac{U_a}{\dot{U_a}^{1-q}}\right]\;\dot{h}+\left[\dot{U_a}^q\right] \;h\\
\label{formule_L2}
&=x^{2-q} \dot{U_a}^q \; \frac{d}{dx}\left[\frac{\dot{h}}{\dot{U_a}^q}\right]+\dot{U_a}^q \;h
\end{align}
and 
\begin{equation}
\label{formule_F}
F(x,h,\dot{h})=\frac{q}{\dot{U_a}^{1-q}}\; h\dot{h}+\left[h\; \dot{U_a}^q+U_a\dot{U_a}^q \right]\left[\left(1+\frac{\dot{h}}{\dot{U_a}}\right)^q-1-q\frac{\dot{h}}{\dot{U_a}} \right].
\end{equation}

We already know from \cite[Proposition 2.1]{Montaru1} that if $t_0>0$,
\begin{equation}
\label{formule_H1}
 h_x(t,x)=\psi(t,x^{\frac{q}{2}}) \text{ \qquad for all } t\geq t_0,\;x\in[0,1]
 \end{equation}
where
\begin{equation}
\label{formule_H2}
\psi\in C^{1,\infty}([t_0,\infty)\times [0,1])
\end{equation}
($\psi$ having odd derivatives vanishing at $x=0$). Formula (\ref{formule_H1}) implies
\begin{equation}
\label{formule_h''}
 \frac{\partial^2 h}{\partial x^2} (t,x)=\frac{q}{2}\frac{\frac{\partial \psi}{\partial x}(t,x^{\frac{q}{2}})}{x^{1-\frac{q}{2}}} \text{\quad for all } t\geq t_0,\,x\in (0,1].
 \end{equation} 
Moreover, from \cite[Theorem 1.2]{Montaru2}, we know that 
\begin{equation}
\label{convergence_C1_h-->0}
\|h(t)\|_{C^1([0,1])}\underset{t\rightarrow +\infty}{\longrightarrow}0.
\end{equation}

We will need the following lemma, whose proof is postponed just after this one~:
\begin{lem}
\label{lem_integrale_Lh_h} Let $t>0$. Denoting $h=h(t)$, we have :
$$ \frac{h\;L_{U_a}h}{x^{2-q}\dot{U_a}^q}\in L^1(0,1)$$
and 
\begin{equation}
\label{integrale_Lh_h}
\int_0^1 \frac{h\;L_{U_a}h}{x^{2-q}\dot{U_a}^q}=-\int_0^1
\left[\frac{\dot{h}^2}{\dot{U_a}^q}-\frac{h^2}{x^{2-q}}\right].
\end{equation}
\end{lem}

By (\ref{convergence_C1_h-->0}), there exists $t_0>0$ such that for all $t\geq t_0$, 
\begin{equation}
\label{condition_h'_petit}
\|\frac{h(t)}{\dot{U_a}}\|_{\infty,[0,1]}\leq \frac{1}{2}
\text{ \qquad and  \qquad} 
\|\dot{h}(t)\|_{\infty,[0,1]}\leq \min \left(1,\frac{2\epsilon}{Ka^{1+q}},\frac{\delta}{K}\right).
\end{equation}

Recalling (\ref{formule_F}), there exists $K>0$ such that 
\begin{equation}
\label{defintion_K}
\left|\frac{h\; F(x,h,\dot{h})}{\dot{U_a}^q}\right|\leq K\left[h^2|\dot{h}|+h^2\dot{h}^2+U_a |h| \dot{h}^2\right] 
\end{equation}
for all $(x,t)\in[0,1]\times [t_0,\infty)$ ($h$ and $\dot{h}$ depending on $t$).

Let $\delta>0$ such that $\lambda+\delta\in(0,\lambda_1-1)$ and
$\epsilon=\frac{\lambda_1-1-\lambda-\delta}{\lambda_1}>0$ which satisfies
\begin{equation}
\label{equation_epsilon}
(1-\epsilon)(\lambda_1-1)-\epsilon=\lambda+\delta.
\end{equation}

It is easy to see that for any $t\geq t_0+1$ there exists $M(t)>0$ such that  
$$\underset{[t-1,t+1]}{\sup}\|h_t\|_{\infty,[0,1]}\leq M(t).$$ 
This follows from (\ref{equation_h})(\ref{formule_L1})(\ref{formule_h''})(\ref{formule_F}) since for all $t\geq t_0$, $\|\dot{h}(t)\|_{\infty,[0,1]}\leq 1$ by (\ref{condition_h'_petit}).\\

Let $t\geq t_0+1$. From now on, we denote $h=h(t)$. We want to 
differentiate $\gamma(t)$ under the integral sign  by applying Lebesgue's dominated theorem. This is allowed since  
$$\left|\frac{h_t\; h}{x^{2-q}\dot{U_a}^q}\right|\leq \|h_t\|_{\infty,[0,1]}\|\dot{h}\|_{\infty,[0,1]}\frac{1}{x^{1-q}\;\dot{U_a}^q(1)}
\leq \frac{M(t)}{x^{1-q}\;\dot{U_a}^q(1)}.$$

Hence, we have :
\begin{align*}
\dot{\gamma}(t)&=2\int_0^1 \frac{h_t\; h}{x^{2-q}\dot{U_a}^q}
=2\int_0^1 \frac{L_{U_a}h\; h}{x^{2-q}\dot{U_a}^q}+2\int_0^1 \frac{F(x,h,\dot{h})\; h}{x^{2-q}\dot{U_a}^q}\\
 &=-2\int_0^1
\left[\frac{\dot{h}^2}{\dot{U_a}^q}-\frac{h^2}{x^{2-q}}\right]+2\int_0^1 \frac{F(x,h,\dot{h})\; h}{x^{2-q}\dot{U_a}^q}\\
&=-2(1-\epsilon)\int_0^1
\left[\frac{\dot{h}^2}{\dot{U_a}^q}-\frac{h^2}{x^{2-q}}\right]-2\epsilon\int_0^1
\left[\frac{\dot{h}^2}{\dot{U_a}^q}-\frac{h^2}{x^{2-q}}\right]+2\int_0^1 \frac{F(x,h,\dot{h})\; h}{x^{2-q}\dot{U_a}^q}\\
&\leq -2[(1-\epsilon)(\lambda_1-1)-\epsilon]\int_0^1
\frac{h^2}{x^{2-q}}-2\epsilon\int_0^1
\frac{\dot{h}^2}{\dot{U_a}^q}+2\int_0^1 \frac{F(x,h,\dot{h})\; h}{x^{2-q}\dot{U_a}^q}
\end{align*}
where we have used Lemma \ref{lem_integrale_Lh_h} and  Proposition \ref{prop_existence_lambda1}. Moreover, 
by (\ref{defintion_K}), it is easy to see that
\begin{align*}
\label{estimation_integrale_F*h}
\int_0^1 \frac{F(x,h,\dot{h})\; h}{x^{2-q}\dot{U_a}^q} &\leq K \left[\|\dot{h}\|_{\infty,[0,1]}+
\|\dot{h}\|_{\infty,[0,1]}^2\right]\int_0^1
\frac{h^2}{x^{2-q}}
+K\int_0^1 \frac{U_a}{x}\frac{|h|}{x}\,x^q\,\dot{h}^2\\
&\leq  2K \|\dot{h}\|_{\infty,[0,1]}\int_0^1
\frac{h^2}{x^{2-q}}
+Ka \|\dot{h}\|_{\infty,[0,1]}\int_0^1 \dot{h}^2.
\end{align*}
 Hence, coming back to the previous calculation and applying (\ref{equation_epsilon}), we have 
\begin{align*}
 \dot{\gamma}(t)&\leq -2(\lambda+\delta-K\|\dot{h}\|_{\infty,[0,1]})\int_0^1
\frac{h^2}{x^{2-q}}+\left[Ka \|\dot{h}\|_{\infty,[0,1]}-\frac{2\epsilon}{a^q}\right]\int_0^1 \dot{h}^2\\
&\leq  -2\lambda\int_0^1\frac{h^2}{x^{2-q}} \qquad \text{ because of } (\ref{condition_h'_petit})\\
&\leq -2\lambda \dot{U_a}(1)^q\;\gamma(t)
\end{align*} 
Then for all $t\geq t_0$,  $$\gamma(t)\leq C_1 \exp(-2\lambda \dot{U_a}(1)^q t)$$
where $$C_1= \gamma(t_0) \exp(2\lambda \dot{U_a}(1)^q t_0)$$
depends on $\lambda$ and $u_0$.\\
 Since $u_0$ has a derivative at $x=0$, it is clear that for some $\overline{a}$ large enough, $U_{\overline{a}}$ is a supersolution (see the proof of \cite[Lemma 4.1]{Montaru2}). Hence by the comparison principle
 (see \cite[Lemma 4.1]{Montaru1}), we have $u(t,x)\leq U_{\overline{a}}(x)\leq \overline{a}x$ for all $x\in[0,1]$ and $t\geq 0$,
 which implies that $\gamma$ is bounded.
So there exists $C_2=C_2(t_0,u_0)$ such that for all $t\in[0,t_0]$,
 $$\gamma(t)\leq C_2 \exp(-2\lambda \dot{U_a}(1)^q t),$$
 whence the result with $C=\max(C_1,C_2)$ depending on $u_0$ and $\lambda$.
\end{proof}

 \textbf{Remark :} 
we see that $t_0$ depends on $\lambda$ and that  $t_0\rightarrow +\infty$
 as $\lambda\rightarrow \lambda_1$. Hence, since $t_0$ may possibly go to infinity,
then we have no bound on $C$. So, 
 we cannot get the result for $\lambda=\lambda_1$, at least by this way.

\medskip
\begin{proof}[Proof of Lemma \ref{lem_integrale_Lh_h}.]
Fixing  $\epsilon\in (0,1)$, since $h\in C^2((0,1])$ and from formula (\ref{formule_L2}), we see that~:
$$\int_\epsilon^1 \frac{L_{U_a}h\;h}{x^{2-q}\dot{U_a}^q}=\int_\epsilon^1 \frac{d}{dx}\left[\frac{\dot{h}}{\dot{U_a}^q} \right]\;h+\frac{h^2}{x^{2-q}}=-\int_\epsilon^1 \left[\frac{\dot{h}^2}{\dot{U_a}^q}-\frac{h^2}{x^{2-q}}\right]-\frac{h(\epsilon)\dot{h}(\epsilon)}{\dot{U_a}^q(\epsilon)}$$
since $h(1)=0$.
Then, since $h(0)=0$ and $h\in C^1([0,1])$, we have 
$$\frac{h^2}{x^{2-q}}\in L^1(0,1)$$
 and 
 $$\frac{h(\epsilon)\dot{h}(\epsilon)}{\dot{U_a}^q(\epsilon)}\underset{\epsilon \rightarrow 0}{\longrightarrow} 0.$$
 Moreover, 
 $$\frac{d}{dx}\left[\frac{\dot{h}}{\dot{U_a}^q}\right]\in L^1(0,1)$$ since $$\frac{d}{dx}\left[\frac{\dot{h}}{\dot{U_a}^q}\right]=\frac{\ddot{h}}{\dot{U_a}^q}+q \frac{U_a}{x^{2-q}\dot{U_a}}\dot{h},$$ $$\left|\frac{U_a}{x}\right|\leq a$$ and 
because of  (\ref{formule_h''}).
 Finally, we get the result by letting $\epsilon$ go to zero since the Lebesgue's dominated theorem can be applied.
\end{proof}

\section{Convergence with exponential speed in $C^1([0,1])$}

We first give a regularizing estimate from $L$ to $C^1([0,1])$ for problem (\ref{intro_equation_h}).

\begin{lem}
\label{lem_estimation_C1_L2} Let $$N\geq 2,$$ $$0<m<M,$$ $$U_a=U_{a(m)}$$ the unique stationary solution of $(PDE_m)$ (equations (\ref{equ_u_1})-(\ref{equ_u_4})) and $$u_0\in Y_m.$$
Then, there exists $\overline{t}=\overline{t}(u_0)>0$, $T=T(N,u_0)>0$, $C=C(N,u_0)>0$ such that, for all $t_0\geq \overline{t}$ and  $t\in(0,T]$,  
$$\|u(t_0+t)-U_a\|_{C^1([0,1])}\leq \frac{C}{t^\beta}\|u(t_0)-U_a\|_L,$$
where
$$\beta=\beta(N)=1+\frac{N}{4}.$$ 

\end{lem}

Before giving the proof of this lemma, we need to recall some notation and well-known properties of
the Dirichlet heat semigroup on the  open unit ball $B$ of $\mathbb{R}^{N+2}$.
\begin{prop}
\label{prop_semigroupe}
We denote $(S(t))_{t \geq 0}$ the Dirichlet heat semigroup on $B=B(0,1)\subset \mathbb{R}^{N+2}$,
$$C_0(\overline{B})=\left\{f\in C(\overline{B}), \;f=0 \text { on }\partial B \right\} $$
and 
$$C_0^1(\overline{B})=\left\{f\in C^1(\overline{B}), \;f=0 \text { on }\partial B \right\}. $$
For $p>1$, 
\begin{equation}
\label{estimation_S(t)_Lp_W1,p}
\|S(t)f\|_{W^{1,p}(B)}\leq \frac{C}{\sqrt{t}}\|f\|_{L^p(B)} \text{ for all } f\in L^p(B).
\end{equation}

 Moreover, let $p>N+2$. 
 Since there exists $C>0$ such that for all $t>0$, 
$$ \|S(t)f\|_{C_0^1(\overline{B})}\leq \frac{C}{\sqrt{t}}\|f\|_{C(\overline{B})} \text{ for all } f\in C_0(\overline{B})$$
and
$$ \|S(t)f\|_{C(\overline{B})}\leq \frac{C}{t^{\frac{N+2}{2p}}}\|f\|_{L^p(B)} \text{ for all } f\in L^p(B)$$
then there exists $C>0$ such that for all $t>0$, 
\begin{equation}
\label{estimation_S(t)_Lp_C1}
\|S(t)f\|_{C^1(\overline{B})}\leq \frac{C}{t^{\gamma}}\|f\|_{L^p(B)} \text{ for all } f\in L^p(B)
\end{equation}
where $\gamma=\frac{1}{2}+\frac{N+2}{2p}<1$.
\end{prop}

\begin{proof}[Proof of Lemma \ref{lem_estimation_C1_L2}.]

As before, we denote $h(t)=u(t)-U_a$. 
Now, we set 
$$w(t,y)=\frac{u(N^2t,|y|^N)}{|y|^N}$$
$$W_a(y)=\frac{U_a(|y|^N)}{|y|^N}$$
$$f(t,y)=\frac{h(N^2t,|y|^N)}{|y|^N}$$
 for all $t\geq 0$ and $y\in 
\overline{B}$ where 
$$B=B(0,1)\subset \mathbb{R}^{N+2} $$ denotes the open unit ball in 
$\mathbb{R}^{N+2} $. Then $w$ is a radial classical solution of  the following transformed problem 
called $(tPDE_m)$  :
\begin{eqnarray}
\label{*equ_w1}
 & w_t=\Delta w + N^2 w \left(w+\frac{y.\nabla 
w}{N}\right)^q  &\mbox{ on }(0,T]\times \overline{B}\\
\label{*equ_w2} \nonumber
&w(0)=w_0&\\
\label{*equ_w3}\nonumber
&w+\frac{y.\nabla w}{N}\geq 0 & \mbox{ on }(0,T]\times \overline{B}\\
\label{*equ_w4}\nonumber
&w=m&\mbox{ on } [0,T]\times \partial B
\end{eqnarray}
Here, "classical" means that for any  $T>0$,
$$ w\in C([0,T]\times \overline{B}) \bigcap C^{1,2} ((0,T]\times \overline{B}).$$
  
Note also that $W_a$ is a radial stationary solution of $(tPDE_m)$ 
and that 
$$f=w-W_a$$ 
which implies obviously $f=0$ on $\partial B$. \\

All these facts rely on the following calculations relating $h$ to $\tilde{f}$ (and also $u$ to $w$ and $U_a$ to $W_a$), where $$f(t,y)=\tilde{f}(t,|y|) \text{ for all } (t,y)\in [0,+\infty)\times \overline{B}.$$ 
We have   for $0<t\leq T$ and $0<x\leq 1$ :
\begin{eqnarray}
\label{h}
h(t,x)&=&x\; \tilde{f}\left(\frac{t}{N^2},x^{\frac{1}{N}}\right).\\
 \label{derivee_en_t_h} \nonumber
h_t(t,x)&=&\frac{x}{N^2} \; \tilde{f}_{t} 
\left(\frac{t}{N^2},x^{\frac{1}{N}}\right).\\
 \label{derivee_h}
h_x(t,x)&=&\left[\tilde{f}+\frac{r\tilde{f}_r}{N}\right] 
\left(\frac{t}{N^2},x^{\frac{1}{N}}\right)\nonumber \\
&=&\left[f+\frac{y.\nabla f}{N}\right] 
\left(\frac{t}{N^2},x^{\frac{1}{N}}\right).\\
 \label{derivee_seconde_h} 
x^{2-\frac{2}{N}}h_{xx}(t,x)&=&\frac{x}{N^2}\left[\tilde{f}_{rr}+\frac{N+1}{r}\tilde{f}_{r}\right] 
\left(\frac{t}{N^2},x^{\frac{1}{N}}\right)\nonumber\\
&=&\frac{x}{N^2}\; \Delta f \left(\frac{t}{N^2},x^{\frac{1}{N}}\right).\nonumber
\end{eqnarray}

\textbf{Remark:} we would like to mention that problem (\ref{*equ_w1}), which exhibits simple Laplacian diffusion, 
was already used in \cite{Montaru1} to prove existence of a solution for
 $(PDE_m)$ and in \cite{Montaru2} to get some estimates implying relative 
compactness of the trajectories in $C^1([0,1])$. Actually, the solution $u$ was obtained from $w$ by the formula
$$u(t,x)=x\,w(\frac{t}{N^2},x^{\frac{1}{N}}) $$
and $w$ was obtained as a limit of solutions $w^\epsilon$ of approximations of (\ref{*equ_w1}) (because the nonlinearity is non Lipschitz), the regularity of w following from that of $w^\epsilon$ since for $\alpha \in [0,\frac{2}{N})$,  $w^\epsilon$ have a bound in $C^{1+\alpha/2,2+\alpha}$ uniform in $\epsilon$. See section 4.5 in \cite{Montaru1} for more details.\\

Since $\dot{U_a}>0$ and $W_a\in C^1(\overline{B})$, then we have 
\begin{equation}
\label{derivee_W_a_positive}
W_a+\frac{y.\nabla W_a}{N} \text{ positive and bounded on }  \overline{B}.
\end{equation} 
A simple computation shows that
\begin{equation}
\label{equation_f}
f_t=\Delta f+ \Phi(y,f,\nabla f)
\end{equation}
where
\begin{align}
\label{def_Phi}
\Phi(y,f,\nabla f)=&
N^2 f\left[W_a+\frac{y.\nabla W_a}{N}+f+\frac{y.\nabla f}{N}\right]^q\nonumber \\
&+N^2 W_a \left(W_a+\frac{y.\nabla W_a}{N}\right)^q\left[\left(1+\frac{f+\frac{y.\nabla f}{N}}{W_a+\frac{y.\nabla W_a}{N}}\right)^q-1\right].
\end{align}

We  observe that, since $\tilde{f}(r)=\frac{h(r^N)}{r^N}$ 
and $B$ is the unit ball in $\mathbb{R}^{N+2}$, then
$$\|f\|_{L^2(B)}^2=\int_0^1 |S_{N+1}|r^{N+1} \frac{h(r^N)^2}{r^{2N}}\;dr
=\frac{|S_{N+1}|}{N} \int_0^1 \frac{h^2}{x^{2-\frac{2}{N}}}\;dx=
\frac{|S_{N+1}|}{N}\|h\|_L^2$$
hence
$$\|f\|_{L^2(B)}=\sqrt{\frac{|S_{N+1}|}{N}}\|h\|_L.$$ 

Other observation :
by (\ref{convergence_C1_h-->0}), 
we know that for t large enough $\|h(t)\|_{C^1([0,1])}$ is as small as 
desired. Hence, since $h(t,0)=0$, we deduce from (\ref{h}) that $\|f(t)\|_{C(\overline{B})}$ can be
 made as small as we wish for large t and then 
 $\|y.\nabla f(t)\|_{C(\overline{B})}$ also from (\ref{derivee_h}). Hence, there exists $$\overline{t}_0=\overline{t}_0(u_0)>0$$ such that for all $t\geq \overline{t}_0$, for all $y\in \overline{B}$,
 $$\left|\frac{f+\frac{y.\nabla f}{N}}{W_a+\frac{y.\nabla W_a}{N}}\right|\leq \frac{1}{2}.$$
This fact, the boundedness  of $W_a+\frac{y.\nabla W_a}{N}$ on  $\overline{B}$(above and below by a positive constant) and (\ref{def_Phi}) imply, for any $p\geq 2$, the existence of $C>0$ such that for all $t\geq \overline{t}_0$,
\begin{equation}
\label{majoration_Phi1}
\|\Phi(y,f(t),\nabla f(t))\|_{L^p(B)}\leq C \|f(t)\|_{W^{1,p}(B)} 
\end{equation}
and
\begin{equation}
\label{majoration_Phi2}
\|\Phi(y,f(t),\nabla f(t))\|_{C(\overline{B})}\leq C \|f(t)\|_{C^1(\overline{B})}
\end{equation}

\quad\\
We will now use the regularizing effect of the Dirichlet heat semigroup recalled in Properties \ref{prop_semigroupe} to show a similar property for (\ref{equation_f}).\\

\textbf{First step.} We show that for all $p\in (1,+\infty)$, there exists $T_0>0$ and $C>0$ such that
$$\|f(t_0+t)\|_{W^{1,p}} \le Ct^{-1/2} \|f(t_0)\|_p,\ \ \hbox{for all $t_0\ge \bar t_0$ et $t\in (0,T_0]$}
\leqno(A)$$
and 
$$\|f(t_0+t)\|_p \le C \|f(t_0)\|_p,\ \ \hbox{pour tout $t_0\ge \bar t_0$ et $t\in (0,T_0].$}
\leqno(A')$$
Let $t_0\geq \overline{t}$ and $t\geq t_0$.\\
 Since $w$ is a classical solution of $(tPDE_m)$ for $t>0$, then $f$
is a classical solution of (\ref{equation_f}), hence also a mild solution. So, 
\begin{equation}
\label{f_mild_solution}
f(t_0+t)=S(t)f(t_0)+\int_0^t S(t-s) \Phi(y,f(t_0+s),\nabla f(t_0+s))\,ds.
\end{equation}
Then, by (\ref{majoration_Phi1}) and (\ref{estimation_S(t)_Lp_W1,p}) we obtain : 
($C_1$ being a positive constant which my vary from line to line)
$$\|f(t_0+t)\|_{W^{1,p}(B)}\leq \frac{C_1}{\sqrt{t}}\|f(t_0)\|_{L^p(B)}+ \int_0^t \frac{C_1}{\sqrt{t-s}}\|f(t_0+s)\|_{W^{1,p}(B)}ds, $$
from which follows
$$\sqrt{t} \|f(t_0+t)\|_{W^{1,p}(B)}\leq C_1\|f(t_0)\|_{L^p(B)}+\sqrt{t}  \int_0^t \frac{C_1}{\sqrt{s(t-s)}}\sqrt{s}\|f(t_0+s)\|_{W^{1,p}(B)}ds .$$
We notice that $ \int_0^t \frac{ds}{\sqrt{s(t-s)}}=\int_0^1 \frac{dx}{\sqrt{x(1-x)}}$ by the change of variable $x=\frac{s}{t}$.\\
Let $T_0=\frac{1}{4{C_1}^2} $. Denoting 
$$a(T_0)= \underset{t\in(t_0,t_0+T_0]}{\sup} \sqrt{t} \|f(t)\|_{W^{1,p}(B)},$$ we get
$$a(T_0)\leq C_1\|f(t_0)\|_{L^p(B)}+ C_1 \sqrt{T_0}\; a_1(T_0)$$
which, by the choice of $T_0$, gives 
$$a(T_0)\leq 2C_1 \|f(t_0)\|_{L^p(B)}.$$
Hence, for all $t\in(t_0,T_0]$, 
$$\|f(t_0+t)\|_{W^{1,p}(B)}\leq \frac{2C_1}{\sqrt{t}}\|f(t_0)\|_{L^p(B)},$$
which proves $(A)$ and allows thanks to (\ref{f_mild_solution}) and (\ref{majoration_Phi1}) again to get $(A')$. \\

\textbf{Second step.} Let us set $n=N+2$.\\ We show by iteration the existence of $p\in (n,+\infty)$ and $C>0$ independent of the solution $f$ such that
$$\|f(t_0+t)\|_{W^{1,p}(B)} \le Ct^{-1/2-(n/2)(1/2-1/p)}  \|f(t_0)\|_2\ \ \hbox{for all $t_0\ge \bar t_0$ et $t\in (0,T_0].$}
\leqno(B)$$
et
$$\|f(t_0+t)\|_{L^p(B)} \le Ct^{-(n/2)(1/2-1/p)}  \|f(t_0)\|_2\ \ \hbox{for all $t_0\ge \bar t_0$ et $t\in (0,T_0].$}
\leqno(B')$$
Indeed, this is true for $p=2$ thanks to $(A)$ and $(A')$. \\
Assume that $(B)$ and $(B')$ are true for some $p\in[2,+\infty)$.\\
If $p<n$, then we prove that $(B)$ and $(B')$ are true for $p=p^*$, $p^*<+\infty$ beeing the optimal exponent such that we have the following Sobolev imbedding
$$ W^{1,p}(B) \subset L^{p^*}(B).$$
Indeed, we have by $(A)$ and Sobolev embedding that
\begin{align*}
\|f(t_0+t)\|_{W^{1,p^*}} 
&\le C(t/2)^{-1/2} \|f(t_0+(t/2))\|_{p^*} \le C(t/2)^{-1/2} \|f(t_0+(t/2))\|_{W^{1,p}} \cr
&\le C(t/2)^{-1/2}  (t/2)^{-1/2-(n/2)(1/2-1/p)} \|f(t_0)\|_2\cr
&=C  (t/2)^{-1/2-(n/2)(1/2-1/p+1/n)} \|f(t_0)\|_2
=C'  t^{-1/2-(n/2)(1/2-1/p^*)} \|f(t_0)\|_2,
\end{align*}
and, by Sobolev embedding and  (B), we have
$$\|f(t_0+t)\|_{p^*} \le C\|f(t_0+t)\|_{W^{1,p}} \le Ct^{-1/2-(n/2)(1/2-1/p)} \|f(t_0)\|_2= Ct^{-(n/2)(1/2-1/p^*)}  \|f(t_0)\|_2.$$

Iterating this process, we obtain after a finite number of steps 
some $p\in [n,+\infty)$ such that $(B)$ and $(B')$ are true. 
\begin{itemize}
\item If $p>n$, this is the result we wanted.
\item If $p=n$, since $B$ (resp. (B') ) is true for $p=2$ and $p=n$, we can interpolate
between $W^{1,2}(B)$ and $W^{1,n}(B)$ (resp. $L^2(B)$ and $L^n(B)$) and get $(B)$ (resp. (B') ) for some $p_0\in (\frac{n}{2},n)$, which, by an application of the previous process, shows $(B)$ (resp. (B') ) for $p=p_0^*>n$.
\end{itemize}
 
\textbf{Last step.} We can now prove the result by making
 a last iteration. \\
  Coming back to (\ref{f_mild_solution}), by (\ref{majoration_Phi2}) and (\ref{estimation_S(t)_Lp_C1}) we obtain : 
$$\|f(t_0+t)\|_{C^1(\overline{B})}\leq \frac{C}{t^{\gamma}}\|f(t_0)\|_{L^p(B)}+ \int_0^t \frac{C}{\sqrt{t-s}}\|f(t_0+s)\|_{C^1(\overline{B})}ds, $$
from which follows
$$t^\gamma \|f(t_0+t)\|_{C^1(\overline{B})}\leq C\|f(t_0)\|_{L^p(B)}+t^\gamma  \int_0^t \frac{C}{s^\gamma\sqrt{t-s}}s^\gamma\|f(t_0+s)\|_{C^1(\overline{B})}ds .$$
Note that $t^\gamma \int_0^t \frac{ds}{s^\gamma\sqrt{t-s}}=\int_0^1 \frac{dx}{x^\gamma\sqrt{1-x}}\;\sqrt{t}$ which is well defined since $\gamma=\frac{1}{2}+\frac{n}{2p}<1$.\\
Let $T=\frac{1}{4C^2} $. Denoting 
$$b(T)= \underset{t\in(t_0,t_0+T]}{\sup} t^\gamma \|f(t)\|_{C^1(\overline{B})},$$ we get
$$b(T)\leq C\|f(t_0)\|_{L^p(B)}+ C \sqrt{T}\; a(T)$$
which, by the choice of T, gives 
$$b(T)\leq 2C \|f(t_0)\|_{L^p(B)}.$$
Hence, for all $t\in(t_0,T]$, 
$$\|f(t_0+t)\|_{C^1(\overline{B})}\leq \frac{2C}{t^\gamma}\|f(t_0)\|_{L^p(B)},$$ 
which implies by (B') that
$$\|f(t_0+t)\|_{C^1(\overline{B})}\leq 
\frac{C'}{t^{\frac{1}{2}+\frac{n}{2p}}}\|f(t_0+t/2)\|_{L^p(B)}\leq
\frac{C'}{t^{\frac{1}{2}+\frac{n}{4}}}\|f(t_0)\|_{L^2(B)}.$$
This implies the result since 
$$\|h(t_0+t)\|_{C^1([0,1])}\leq  C_1 \|f(t_0+t)\|_{C^1(\overline{B})}\leq \frac{C_2}{t^{\frac{1}{2}+\frac{n}{4}}}\|f(t_0)\|_{L^2(B)}=\frac{C_3}{t^{\frac{1}{2}+\frac{n}{4}}}\|h(t_0)\|_L.$$
\end{proof}

\bigskip

We can now  give the proof of 
Theorem \ref{thm_convergence_expo_C1}.

\begin{proof}[Proof of Theorem \ref{thm_convergence_expo_C1}]
This follows from Lemma \ref{lem_convergence_expo_L2} and from Lemma \ref{lem_estimation_C1_L2}, which is a regularizing in time estimate. Indeed, using notation of Lemma \ref{lem_estimation_C1_L2} (having fixed $p>N$), let $t\geq \overline{t}+T$. Then
$t-T\geq \overline{t}$ so we obtain, $$\|u(t)-U_a\|_{C^1([0,1])}\leq \frac{C}{T^\gamma}\|u(t-T)-U_a\|_L\leq C(u_0,p) \exp(-\lambda\dot{U_a}(1)^q\;(t-T)),$$
which gives the result.
\end{proof}

\section{Appendix : proofs of the preliminary results for dimension $N=2$}
In this section, we first recall some basic facts about continuous dynamical systems and Lyapunov functionals. In the next subsections are the proofs of all results of section \ref{N=2}.

\subsection{Reminder on continuous dynamical systems and Lyapunov functionals}
\label{Lyapunov_rappels}
For reader's convenience, we fast recall some very basic facts on continuous dynamical systems, which are general but will be given in the context of  
$$Y_m^1=Y_m\cap C^1([0,1])$$
endowed with the induced topology of $C^1([0,1])$.
For reference, see \cite[chap. 9]{Haraux}.\\

Here follow the definitions of a continuous dynamical system, its trajectories, stationary points and $\omega$-limit sets. 
\begin{Def}
A continuous dynamical system on $Y_m^1$ is a one-parameter family of mappings $(T(t))_{t\geq 0}$ from $Y_m^1$ to $Y_m^1$   such that :
\begin{itemize}
\item [i)] $T(0)=Id$.
\item [ii)]$T(t+s)=T(t)T(s)$ for any $t,s\geq 0$.
\item [iii)]For any $t\geq 0$, $T(t)\in C(Y_m^1,Y_m^1)$.
\item [iv)]For any $u_0\in Y_m$, $t\mapsto T(t)u_0 \, \in C((0,\infty),Y_m^1)$.
\end{itemize}
\end{Def}

\begin{Def}
 Let $u_0\in Y_m^1$.
\begin{itemize}
\item $u_0$ is a stationary point  if for all $t\geq 0$, $T(t)u_0=u_0$.
 \item  $\gamma_1(u_0)=\{T(t)u_0,\; t\geq 1 \}$ is the 
 trajectory of $u_0$ from $t=1$.
\item  $\omega(u_0)=\{v\in Y_m^1,\;
\exists t_n\rightarrow +\infty,\; t_n\geq 1,\; T(t_n)u_0 \underset{n \rightarrow +\infty}{\longrightarrow}v \mbox{ in }Y_m^1\}$\\ is the $\omega$-limit set of $u_0$.
\end{itemize}
\end{Def}

Now we give the definition of a strict Lyapunov functional and  Lasalle's invariance principle.

\begin{Def}\quad 
\begin{itemize}
\item [i)]$\mathcal{F}\in C(Y_m^1,\mathbb{R})$ is a Lyapunov functional if for all $u_0\in Y_m^1$,
$$t\mapsto \mathcal{F}[T(t)u_0] \mbox{ is nonincreasing on } [0,+\infty) .$$
\item [ii)]A Lyapunov functional $\mathcal{F}$ is a strict Lyapunov functional if 
$$ \mathcal{F}[T(t)u_0] =\mathcal{F}[u_0] \mbox{ for all } t\geq 0  \mbox{ implies that $u_0$ is an equilibrium point.}$$
\end{itemize}
\end{Def}

\begin{propo} \label{equilibre} Lasalle's invariance principle.\\
Let $u_0\in Y_m^1$.
Assume that the dynamical system $(T(t))_{t\geq 0}$ admits a strict Lyapunov functional
 and that $\gamma_1(u_0)$ is relatively compact in $Y_m^1$.\\
Then the $\omega$-limit set $\omega(u_0)$ is nonempty and consists of equilibria of the dynamical system.
\end{propo}

See \cite[p. 143]{Haraux} for a proof.

\subsection{Wellposedness and regularity for problem $(PDE_m)$}
We first remark that there is a classical comparison 
principle  available for problem $(PDE_m)$, 
which will for instance imply the uniqueness of 
the maximal classical solution in Theorem \ref{thm_existence_u}. 
 
\begin{lem}
\label{PC_PDE}
Let $T>0$. Assume that :
\begin{itemize}
\item $u_1,u_2\in C([0,T]\times [0,1])\bigcap
C^1((0,T]\times [0,1])\bigcap C^{1,2}((0,T]\times (0,1)).$ 
\item For all $t\in (0,T] $, $u_1(t)$ and $u_2(t)$ are nondecreasing. 
\item There exists $i_0\in\{1,2\}$ and some $\gamma<1$ such that 
\begin{equation}
\label{majoration_u_i0}
 \underset{t\in(0,T]}{\sup} t^\gamma\; \|u_{i_0}(t)\|_{C^1([0,1])}<\infty .
\end{equation}
\end{itemize}
Suppose moreover that :
\begin{eqnarray} 
 & {u_1}_t\leq x\;{u_1}_{xx}+u_1 {u_1}_x  & \mbox{ for all }(t,x)\in (0,T]\times (0,1)\\
 & {u_2}_t\geq x\;{u_2}_{xx}+u_2 {u_2}_x  & \mbox{ for all }(t,x)\in (0,T]\times (0,1)\\
& u_1(0,x)\leq u_2(0,x) &\mbox{ for all }x\in [0,1]\\
& u_1(t,0)\leq u_2(t,0) &\mbox{ for }t\geq 0 \\
& u_1(t,1)\leq u_2(t,1) &\mbox{ for }t\geq 0
\end{eqnarray}
Then $u_1\leq u_2$ on $[0,T]\times [0,1]$.\\
\end{lem}

The proof of this result was given in \cite{Souplet-Nikos} under weaker assumptions. We give a different one in this simpler context.

\begin{proof}[Proof of Lemma \ref{PC_PDE}]
Let us set $$z=(u_1-u_2)e^{-\int_0^t (\|u_{i_0}(s)\|_{C^1}+1)ds},$$
well defined thanks to (\ref{majoration_u_i0}).
 The hypotheses made show that
$$z\in C([0;T]\times [0;1])\bigcap C^1((0,T]\times [0,1])\bigcap C^{1,2}((0,T]\times (0,1)).$$ 
Assume now by contradiction that $\underset{[0;T]\times [0;1]}{\max} z>0$. \\
By assumption, $z\leq 0$ on the parabolic boundary of $[0,T]\times [0,1]$. \\
Hence,
$\underset{[0;T]\times [0;1]}{\max} z$ is reached at a point $(t_0,x_0)\in (0;T]\times (0;1)$. \\Then $z_x(t_0,x_0)=0$ so  
$(u_1)_x(t_0,x_0)=(u_2)_x(t_0,x_0)$. \\
Moreover, $z_{xx}(t_0,x_0)\leq 0$ and $z_t(t_0,x_0)\geq 0$.
 But we have $$z_t(t_0,x_0)\leq x\,z_{xx}(t_0,x_0)+ \left[ (u_{i_0})_x(t_0,x_0)-\|u_{i_0}(t_0)\|_{C^1}-1\right] z(t_0,x_0).$$
 The LHS of the inequality is nonnegative and the RHS is negative, whence the 
 contradiction.
\end{proof}

Before coming to the proof of Theorem \ref{thm_existence_u}, we need to fix some notation and recall some facts about the Dirichlet heat semigroup.

For reference, see for instance the book \cite{Lunardi} of A. Lunardi.

\begin{nota} \quad 
\begin{itemize}
\item $B$ denotes the open unit ball in $\mathbb{R}^{4}$.
\item $Z_0=\{W \in C(\overline{B}), \; 
W | _{\partial B}=0 \}$.
\item $(S(t))_{t\geq 0}$ denotes the heat semigroup  on $Z_0$. It is the restriction on $Z_0$ of the Dirichlet heat semigroup on $L^2(B)$.
\item $(X_\theta)_{\theta \in [0,1]}$ denotes the scale of interpolation spaces for 
$(S(t))_{t\geq 0}$, where $X_0=Z_0$, $X_1=D(-\Delta)$ and $X_\alpha \hookrightarrow X_\beta$ with dense continuous injection for any $\alpha>\beta$, $(\alpha,\beta)\in[0,1]^2$.
\end{itemize}

\end{nota}

\begin{prop}\quad 
\begin{itemize}
\item $X_{\frac{1}{2}}=\{ W \in C^1(\overline{B}), \;W | _{\partial B}=0 \} . $
\item Let $\gamma_0 \in (0;\frac{1}{2}]$. For any $\gamma\in[0,2\gamma_0)$,  $$X_{\frac{1}{2}+\gamma_0} \subset C^{1,\gamma}
(\overline{B})$$  with continuous embedding.
\item There exists $C_D\geq 1$ such that for any $\theta \in [0;1]$,
 $W \in X_0$ and $t>0$,
$$\|S(t)W \|_{X_\theta}\leq \frac{C_D}{t^{\theta}}\|W\|_{\infty}. $$
\end{itemize}

\end{prop}

We just want to introduce some specific notation we are going to use.
\begin{nota} 
Let $(a,b)\in(0,1)^2$. We denote $I(a,b)=\int_0^1 \frac{ds}{(1-s)^as^b}$.  \\ For all $ t\geq 0,\;
\int_0^t \frac{ds}{(t-s)^as^b}=t ^{1-a-b}I(a,b) .$

\end{nota}

\begin{nota} Let $m\geq 0$ and $\gamma>0$. 
\begin{itemize}
\item $Y_m=\{u\in C([0;1]) \mbox{ nondecreasing},\; u'(0) \mbox{ exists, } u(0)=0,u(1)=m\}$.
\item $Z_m=\{w\in C(\overline{B}), \; w| _ { \partial B } = m\}$.
\item  $Y_m^{1,\gamma}=\{u\in Y_m\cap C^1([0,1]), \;\underset{x\in (0,1]}{\sup}\; \frac{|u'(x)-u'(0)|}{x^\gamma}<\infty\}$.
\item $Z_m^{1,\gamma}=\{w\in Z_m\cap C^1(\overline{B}),\; 
\underset{y \in \overline{B}\backslash \{0\}}{\sup}\; \frac{|\nabla w(y)|}{ |y|^{\gamma}}<\infty \}$.
\end{itemize}
\end{nota}

\begin{proof}[Proof of Theorem \ref{thm_existence_u}] We begin by giving  a short proof of points i)ii)iii)iv).\\

We define the following transformation $\theta_0$, already remarked in \cite[section 2.2]{Brenner} and \cite[section 2]{Guerra}, and also used in \cite{Souplet-Nikos} :

 $$\begin{array}{ll}
\theta_0 : &Y_m\longrightarrow Z_m \\ 
&u\longrightarrow w \mbox{ where }w(y)= \frac{u(|y|^2)}{|y|^2}\mbox{ for all } y\in 
\overline{B}\backslash \{0\} .
\end{array}$$

The next lemma has been proved in \cite[Lemma 4.3]{Montaru1}.
\begin{lem} 
\label{propriete_theta}
 Let $m\geq 0$.
\begin{itemize}
\item[i)] $\theta_0$ sends $Y_m$ into $Z_m$.
\item [ii)]If $\gamma>\frac{1}{2}$, then  
$\theta_0$ sends $Y_m^{1,\gamma}$ into $Z_m^{1,2\gamma-1 }$.
\end{itemize}
\end{lem} 

If $u_0\in Y_m$, we set  $$w_0=\theta_0(u_0)\in Z_m$$ and $$w(t,y)=\frac{u(4t,|y|^2)}{|y|^2}$$ for all $y\in 
\overline{B}$. \\
Then we obtain a transformed problem 
called $(tPDE_m)$ with simple Laplacian diffusion in $B\subset \mathbb{R}^4$ :
\begin{eqnarray}
\label{equ_w1}
 & w_t=\Delta w + 4 w \left(w+\frac{y.\nabla 
w}{2}\right)  &\mbox{ on }(0,T]\times \overline{B}\\
\label{equ_w2}
&w(0)=w_0&\\
\label{equ_w3}
&w+\frac{y.\nabla w}{N}\geq 0 & \mbox{ on }(0,T]\times \overline{B}\\
\label{equ_w4}
&w=m&\mbox{ on } [0,T]\times \partial B
\end{eqnarray}

This relies on the following calculations relating $u$ to $\tilde{w}$, where we denote $$w(t,y)=\tilde{w}(t,|y|) \text{ for all } (t,y)\in [0,+\infty)\times \overline{B}.$$ 
We have   for $0<t\leq T$ and $0<x\leq 1$ :
\begin{eqnarray}
\label{u}
u(t,x)&=&x\; \tilde{w}\left(\frac{t}{4},\sqrt{x}\right).\\
 \label{derivee_en_t_u} \nonumber
u_t(t,x)&=&\frac{x}{4} \; \tilde{w}_{t} 
\left(\frac{t}{4},\sqrt{x}\right).\\
 \label{derivee_u}
u_x(t,x)&=&\left[\tilde{w}+\frac{r\tilde{w}_r}{2}\right] 
\left(\frac{t}{4},\sqrt{x}\right)\nonumber \\
&=&\left[w+\frac{y.\nabla w}{2}\right] 
\left(\frac{t}{4},\sqrt{x}\right).\\
 \label{derivee_seconde_u} 
x\,u_{xx}(t,x)&=&\frac{x}{4}\left[\tilde{w}_{rr}+\frac{3}{r}\tilde{w}_{r}\right] 
\left(\frac{t}{4},\sqrt{x}\right)\nonumber\\
&=&\frac{x}{4}\; \Delta w \left(\frac{t}{4},\sqrt{x}\right).\nonumber
\end{eqnarray}

 The existence of a unique maximal classical solution w on $[0,T^*)$ of problem $(tPDE_m)$ with initial condition $w_0\in Z_m$, i.e.    a function 
$$ w\in C([0,T^*)\times \overline{B}) \bigcap C^{1,2} ((0,T^*)\times \overline{B})$$
satisfying
 $(\ref{equ_w1})(\ref{equ_w2})(\ref{equ_w3})(\ref{equ_w4})$  is standard. \\
 Indeed, we can set $W=w-m$, get a corresponding equation for $W$, obtain by a fixed point argument  a mild solution W on $[0,\tau^*]$ for some small $\tau^*>0$  by use of the Dirichlet heat semigroup  since the nonlinearity is locally Lipschitz in $(w,\nabla w)$ on  $C^1(\overline{B})$, and finally exploit regularity results to prove that the solution is classical. \\
 Moreover, 
\begin{equation}
\label{majoration_w_C1}
 \underset{t\in(0,\tau^*]}{\sup} \sqrt{t} \|w(t)\|_{C^1(\overline{B})} <\infty
 \end{equation}
 Again by iteration of regularity results on (\ref{equ_w1}),  it can also be proved that 
$$ w\in C^{\infty}((0,T^*)\times \overline{B}).$$
Since $\tau^*=\tau^*(\|w_0\|_{\infty, \overline{B}})$, we also get the blow-up alternative
$$T^*=+\infty \quad \text{ or }\quad\underset{t\rightarrow T^* }{\lim}
\|w(t)\|_{\infty, \overline{B}}=+\infty.$$

Now, we  come back to problem $(PDE_m)$. Let $u_0\in Y_m$ and
 $w_0=\theta_0(u_0)$. By uniqueness, $w$ is radial because $w_0$ is.
   Then, if we set 
 $$T_{max}(u_0)=4 \,T^*(w_0).$$
and  for all $(t,x)\in [0,T_{max})\times [0,1]$, 
\begin{equation}
\label{def_u}
u(t,x)=x\; \tilde{w}(\frac{t}{4},\sqrt{x}),
\end{equation}
then we can check that $u$ is the unique maximal classical solution 
of problem $(PDE_m)$ with initial condition $u_0$. The fact that $u$ is nondecreasing on $[0,1]$ will be shown in vii). It is easy to see that the small existence time $\tau^*=\tau^*(\|w_0\|_{\infty, \overline{B}})$ for w gives a small existence time $\tau=\tau(\mathcal{N}[u_0])$ for u, i.e.  for each $K>0$, 
there exists $\tau=\tau(K)>0$ such that if $\mathcal{N}[u_0]\leq K$ then the solution u is at least defined on $[0,\tau]$. Moreover, the regularity of w implies the results of regularity on u.
Hence, we have proved i)ii)iii) and iv).\\

v) If $u_0\in Y_m^{1,\gamma}$ with $\gamma>1/2$, then 
from Lemma \ref{propriete_theta} ii), 
$$w_0=\theta_0(u_0)\in Z_m^{1,2\gamma-1}\subset C^1(\overline{B}).$$ 
We only have to check the continuity 
at $t=0$ of $t\mapsto w(t)\in C^1(\overline{B})$.
This is clear by the variation of constants formula since $t\mapsto S(t)\Phi\in 
C([0,+\infty),X_{\frac{1}{2}})$  for any $\Phi\in X_{\frac{1}{2}} $.
Hence,
 we get
 a maximal classical solution $$w\in C([0,T^*),C^1(\overline{B}))$$ 
which, thanks to formula (\ref{derivee_u}), gives a maximal classical
 solution of  $(PDE_m)$ $$ u\in C([0,T_{max}),C^1([0,1])).$$

vi) Let $(t,x)\in (0,T_{max}(u_0))\times (0,1]$. From formulas (\ref{u}) and (\ref{derivee_u}), we have
$$u(t,x)=x\; \tilde{w}\left(\frac{t}{4},\sqrt{x}\right)$$
and
$$u_x(t,x)=\tilde{w}\left(\frac{t}{4},\sqrt{x}\right)+\sqrt{x}\; \frac{\tilde{w}_r}{2}
\left(\frac{t}{4},\sqrt{x}\right).$$
These formulas  allow to prove that $u(t)\in C^1([0,1])$ with 
$u_x(t,0)=\tilde{w}(\frac{t}{4},0)$. Since $w(\frac{t}{4})$ is radial, then 
$\tilde{w}_r(\frac{t}{4},0)=0$. This implies  that for any $y\in [0,1]$,  
$$\left|\tilde{w}\left(\frac{t}{4},y\right)-\tilde{w}\left(\frac{t}{4},0\right)\right|\leq  K \;\frac{y^2}{2}$$ and
$$\left|\tilde{w}_r\left(\frac{t}{4},y\right)\right|\leq K\; y $$
where $K= \|\tilde{w}(\frac{t}{4})_{rr}\|_{\infty,[0,1]} $. 
So, we obtain
$$ |u_x(t,x)-u_x(t,0)|\leq K\, x .$$
Hence, $u(t)\in Y_m^{1,1}$.\\

 vii) Let us now show that $u_x(t,x)>0$  for all $(t,x)\in (0,T_{max})\times [0,1]$.\\

We prove the result in two steps. Let $T\in(0,T_{max}).$\\
 \underline{First step :}
 We now show that $v:=u_x \geq 0$ on $(0,T]\times 
 [0,1]$. \\
 We divide the proof in three parts.
 
 \begin{itemize}
 \item 
  \underline{First part :} We show the result for any $u_0\in Y_m^{1,\gamma}$ where 
  $\gamma>\frac{1}
 {2}$.\\
Since $u$ satisfies on $(0,T]\times (0,1]$
\begin{equation}
\label{equation_u_epsilon}
u_t=x\,u_{xx}+u\, u_x
\end{equation}
and thanks to iv), 
we can now differentiate this equation with respect to $x$. We denote 
$$b=1+u $$
  and obtain the partial 
differential equation satisfied by $v$ :
\begin{eqnarray}
\label{v_app_1}
 v_t=x\,v_{xx}+b\;v_x+v^2 &\mbox{ on }&
(0,T)\times(0,1)\\
\label{v_app_2}
 v(0,\cdotp)=(u_0)'&&\\
 \label{v_app_3}
v(t,0)=u_x(t,0) &\mbox{   for  }& t\in(0,T] \\
\label{v_app_4}
v(t,1)=u_x(t,1) &\mbox{        for } &t\in(0,T] 
\end{eqnarray}

By  vi), we know that $u\in C([0,T], C^1( 
[0,1]))$, then $v \in C([0,T]\times [0,1])$
and $v$ reaches its minimum on $[0,T]\times [0,1]$.\\ 
By comparison principle, we have $$0\leq u\leq m$$ so 
$$u_x(t,0)\geq 0$$ and $$u_x(t,1)\geq 0$$ for all $t\in(0,T]$.
Then, from $(\ref{v_app_2}),(\ref{v_app_3})$ and $(\ref{v_app_4})$, $v\geq 0$ on 
the parabolic boundary of $[0,T]\times [0,1]$. From $(\ref{v_app_1})$, we see that 
$v$ cannot reach a negative minimum 
in
$(0,T]\times (0,1)$. So 
$v \geq 0$ on $[0,T]\times [0,1]$.
\item
 \underline{Second part :} We show that if $u_0\in Y_m$, there exists $\tau\in (0,T)$ 
 such that for all
 $t\in [0,\tau]$,  $u(t)$ is nondecreasing on $[0,1]$.\\
 Let $u_0\in Y_m$. By Lemma 4.4 in \cite{Montaru1}, there exists a sequence 
 $(u_{0,n})_{n\geq 1}$  of $ Y_m^{1,1}$  such that 
 $$\|u_{0,n}-u_0\|_{\infty,[0,1]} 
 \underset{n\rightarrow \infty}{\longrightarrow} 0$$ and
$$\mathcal{N}[u_{0,n}]\leq \mathcal{N}[u_0].$$
Since $\mathcal{N}[u_{0,n}]$ is bounded, we know by ii) that 
  there exists a common small existence time $\tau\in(0,T)$ for all solutions $(u_n(t))_{ 
  t\geq 0}$ of problem $(PDE_m)$ with initial condition $u_{0,n}$.
 From first part, we know that for all
 $t\in [0,\tau]$ $u_n(t)$ is a nondecreasing function since $u_{0,n}\in Y_m^{1,1}$. 
 To prove the result, it is sufficient to show that $$\|u_n-u\|_{\infty,
 [0,1]\times [0, \tau]} 
 \underset{n\rightarrow \infty}{\longrightarrow} 0.$$
 Let $\eta>0$. By $(\ref{maj_u_C1})$, there exists $C>0$
 such that for all $t\in [0,\tau]$, $\|u(t)_x\|_\infty \leq 
 \frac{C}{\sqrt{t}}$. So we can choose
  $\eta'>0$ such that $$\eta' e^{\int_0^\tau [\|u(t)_x\|_\infty+1] \;dt}\leq 
  \eta$$
 Let $n_0\geq 1$ such that for all $n\geq n_0$, $\|u_{0,n}-u_0\|_{\infty,[0,1]}\leq \eta' $. 
 Let $n\geq n_0$.\\
 Let us set $$z(t)=[u_n(t)-u(t)]e^{-\int_0^\tau [\|u(t)_x\|
 _\infty+1] \;dt}$$ 
 We see that $z$ satisfies
 \begin{equation}
 \label{equ_aux_z}
  z_t=x\,z_{xx}+b\, z_x+c\, z
 \end{equation}
 
 where $b=u_n(t)$ and 
  $c=[u_x- \|(u)_x\|_\infty-1]<0$.\\
  Since $z\in C([0,\tau]\times [0,1])$, $z$ reaches its maximum and its minimum. \\
  Assume that this maximum is greater than $\eta'$. Since $z=0$ for $x=0$ and $x=1$ and 
  $z\leq \eta'$ for $t=0$, it can be reached only in 
  $(0,\tau]\times (0,1)$ but this is impossible because $c<0$ and (\ref{equ_aux_z}). 
  We make the similar reasoning for the minimum.
   Hence,
  $|z|\leq \eta'$ on  $[0,\tau]\times [0,1]$. \\Eventually,
  $\|u_n-u\|_{\infty,[0,1]\times [0, \tau]} \leq \eta'
   e^{\int_0^\tau [\|u(t)_x\|_\infty+1] \;dt}\leq \eta$ for all $n\geq n_0$. 
   Whence the result.
   
   \item
 \underline{Last part :} Let $u_0\in Y_m$. From the second part, there exists $\tau\in(0,T)$ 
 such that that for all $t\in[0,\tau]$, $u(t)$ is nondecreasing. Since 
 $u\in C([\tau,T],C^1([0,1])  )$ and $u(\tau)\in Y_m^{1,1}$, we can apply the same argument as in the first part to deduce that for 
 all $t\in [\tau,T]$, $u(t)$ is nondecreasing. This concludes the 
 proof of the first step.
 \end{itemize}

 \underline{Second step :} Let us show that $v>0$ on $(0,T]\times [0,1]$.\\
 0 is clearly a subsolution of problem $(tPDE_m)$ so $w\geq 0$ on $\overline{B}$ but by strong maximum principle we even have $$w>0$$ on $B$ (see \cite[Theorem 5 p.39]{Friedman}).
Then, from formula (\ref{derivee_u}) it follows that  $$v(t,0)=u_x(t,0)>0 $$ 
for 
$t\in(0,T]$.\\
Assume by contradiction that $v$ is zero at some point
 in $(0,T)\times (0,1)$. \\
 Since $v$ satisfies (\ref{v_app_1}) and the underlying operator is
 parabolic on  $(0,T)\times (0,1]$, by the strong minimum principle
 (see \cite[Theorem 5 p.39]{Friedman}),   
we deduce that $v=0$ on $(0,T)\times (0,1)$.
Then, by continuity, $v(t,0)=0$ for $t\in(0,T)$ which contradicts the previous 
assertion.\\
Suppose eventually that $v(t,1)=0$ for some $t\in(0,T)$. From $ 
(\ref{equation_u_epsilon})$, we 
deduce
that $u_{xx}(t,1)=0$, ie $$v_x(t,1)=0.$$ 
Since  $v^2\geq 0$, we observe that $v$ 
satisfies :
 \begin{equation}
v_t\geq x\,v_{xx}+[1+u]v_x
\end{equation} 
Since $v>0$ on $(0,T)\times [\frac{1}{2},1)$ and the underlying operator in the 
above 
equation is uniformly parabolic on
$(0,T)\times [\frac{1}{2},1]$, we can apply Hopf's minimum principle (cf. \cite[Theorem 
3, p.170]{Protter}) to deduce that 
$v_x(t,1)<0$ what yields a contradiction. In conclusion, $u_x>0$ on 
$(0,T]\times [0,1]$ 
for all $T<T_{max}$, whence the result.

\end{proof}

\subsection{Subcritical case : Lyapunov functional and convergence in $C^1([0,1])$}

Here are the proofs of results in subsection \ref{N=2-convergence}.

\begin{proof}[Proof of Lemma \ref{traj_globales}]
$m=0$ is trivial so we assume $0<m<2$.\\
 Let $T_{max}=T_{max}(u_0)$.\\
From Theorem $\ref{thm_existence_u}$, in order to get $T_{max}=+\infty$, it is sufficient to prove that 
$$\underset{t\in[0,T_{max})}{\sup}\mathcal{N}[u(t)]<\infty .$$
This fact  easily follows from a comparison with a supersolution of problem $(PDE_m)$.
The main idea is that since $m<2$, if $a_0$ is large enough then $$u_0\leq U_{a_0}$$ and $U_{a_0}$ is then 
a supersolution so for all $t\in [ 0,T_{max})$, $0\leq u(t)\leq U_{a_0}$  hence $$\mathcal{N}[u(t)]\leq a_0$$ since $U_{a_0}$ is concave. \\
Now we give an explicit formula for $a_0$ which will end the proof.
We denote $$a=\frac{m}{1-\frac{m}{2}}$$ which defines the unique steady state, i.e. satisfying $U_a(1)=m$.\\
 First, since $u_0$ is differentiable at $x=0$, $x\mapsto \frac{u_0(x)}{x}$ can be extended continuously to $[0;1]$, so $m\leq \mathcal{N}[u_0]<+\infty$.\\
Let us set $x_0=\frac{m}{\mathcal{N}[u_0]}\in (0,1]$. We can check that for 
 $$a_0= \frac{a}{x_0}$$ we have $$U_{a_0}(x_0)=U_{a_0x_0}(1)=m.$$ 
- For $x\in [0;x_0]$, $u_0(x)\leq \mathcal{N}[u_0] x\leq U_{a_0}(x)$ since by concavity, $U_{a_0}$ is above its chord 
between $x=0$ and $x=x_0$. \\
- For $x\in[x_0,1]$, $u_0(x)\leq m=U_{a_0}(x_0)\leq U_{a_0}(x)$ since 
$U_{a_0}$ is increasing. \\
Hence, $u_0\leq 
U_{a_0}$ on $[0,1]$ where 
$$a_0=\frac{\mathcal{N}[u_0]}{1-\frac{m}{2}} $$

\textbf{Remark : }we actually  proved the following stronger result, to be used in the next proof.\\
For each $K>0$, for any $u_0\in Y_m$ with  $\mathcal{N}[u_0]\leq K$, we have  
$$\underset{t\in [0,\infty)}{\sup}\mathcal{N}[u(t)]\leq \frac{K}{1-\frac{m}{2}}.$$
\end{proof}

\medskip

\begin{proof}[Proof of Lemma \ref{lem_compacite}]
Actually, we will prove the following stronger result :\\
If $m<2 $, $\gamma\in [0;1)$, $t_0>0$ and $K>0$, then
there exists $D_K>0$ such that for any $u_0\in Y_m $ with $\mathcal{N}[u_0]\leq K$, 
we have $$\underset {t \geq t_0}{\sup \;} \| 
 u(t) \| _{C^{1,\frac{\gamma}{N}}} \leq D_K.$$

Let $u_0\in Y_m $ such that $\mathcal{N}[u_0]\leq K$. 
Let $w_0=\theta_0(u_0)$.\\
\underline{First step :} thanks to the final remark in the proof of Lemma \ref{traj_globales}, we have
$$\underset{t\in [0,\infty)}{\sup}\mathcal{N}[u(t)]\leq C_K:=\frac{K}{1-\frac{m}{2}}.$$
Since for $t\geq0$, $\|w(t)\|_{\infty,\overline{B}}=\mathcal{N}[u(\frac{t}{4})]$, we 
deduce that $w$ is global and that 
$$\underset{t\in [0,\infty)}{\sup}\|w(t)\|_{\infty,\overline{B}}=\underset{t\in [0,\infty)}
{\sup}\mathcal{N}[u(t)]\leq C_K.$$

 \underline{Second step :} Let $$\tau=\frac{t_0}{4} $$ and $t\in [0,\tau]$. \\
Denoting $W_0=w_0-m$, then
 \begin{equation}
 \label{w_mild}
 w(t)-m=S(t)W_0+4\int_0^t S(t-s)w \left(w+\frac{x.\nabla 
 w}{2}\right) ds,
 \end{equation}
 so
 $$ \|w(t)\|_{C^1}\leq m+ \frac{C_D}{\sqrt{t}}(C_K+m) +4\int_0^t \frac{C_D}{\sqrt{t-s}}C_K 
 \|w(s)\|_{C^1} ds .$$
 Setting $h(t)=\underset{s\in (0,t]}{\sup} \sqrt{s}\|w(s)\|_{C^1}$, we have $h(t)<\infty$ by (\ref{majoration_w_C1}) and 
 $$ \sqrt{t}\|w(t)\|_{C^1}\leq m \sqrt{\tau}+ C_D(C_K+m) +4C_KC_D\sqrt{t}\int_0^t \frac{1}
 {\sqrt{s}\sqrt{t-s}} h(s) ds ,$$
 $$ \sqrt{t}\|w(t)\|_{C^1}\leq m \sqrt{\tau}+ C_D(m+C_K) +4C_KC_D I\left(\frac{1}{2},\frac{1}{2}\right) 
 \sqrt{t}\; h(t) .$$
 Let $T\in(0,\tau]$. Then,
 \begin{equation}
 \label{h(T)}
  h(T)\leq m \sqrt{\tau}+ C_D(m+C_K) +4C_KC_D I\left(\frac{1}{2},\frac{1}{2}\right) \sqrt{T}\;h(T) .
   \end{equation}
Setting $A=m \sqrt{\tau}+ C_D(m+C_K)$ and $B=8C_K C_D I \left( \frac{1}{2},\frac{1}{2} \right)  $, assume that 
there exists $T\in [0,\tau]$ such that $$h(T)=2A.$$ Then,
$$2A\leq A+\frac{B}{2}\sqrt{T}\;2A \mbox{ which implies } T\geq \frac{1}{B^2}.$$
Let us set $$\tau'=\min\left(\tau,\frac{1}{2B^2}\right).$$ 
Since $h\geq 0$ is nondecreasing, $h_0=\underset{t\rightarrow 0^+}{\lim} h(t)$ exists and  $h_0\leq A$ by (\ref{h(T)}). So by continuity of $h$ on $(0,\tau']$, 
$ h(t)\leq 2A \mbox{ for all }t\in(0,\tau'] $, that is to say :
$$ \|w(t)\|_{C^1}\leq \frac{2A}{\sqrt{t}} \mbox{ for all }t\in(0,
\tau'],$$ 
where $A$ and $\tau'$ only depend on $K$. Then, setting $A_K=2A$, we have
$$\underset{t\in[0,\tau']}{\sup}\sqrt{t}\|w(t)\|_{C^1}\leq A_K.$$

\underline{Third step :}
Let $\gamma_0\in (\frac{\gamma}{2},\frac{1}{2})$ and $t\in[0,\tau']$.\\
Setting $W=w-m$ and $W_0=w_0-m$, then for $t\geq 0$, due to (\ref{w_mild}), we get 
\begin{align*} 
  \|W(t)\|_{X_{\frac{1}{2}+\gamma_0}}&\leq  \frac{C_D}{t^{\frac{1}{2}+\gamma_0}}(C_K+m) +4\int_0^t 
\frac{C_D}{(t-s)^{\frac{1}{2}+\gamma_0}}C_K 
 \frac{A_K}{\sqrt{s}} ds .&
   \end{align*}
 Then we deduce that :
 \begin{align*} 
  t^{\frac{1}{2}+\gamma_0}\|W(t)\|_{X_{\frac{1}{2}+\gamma_0}}&\leq  C_D(C_K+m) +4C_K C_D A_K t^{\frac{1}{2}+\gamma_0}\int_0^t \frac{1}
 {(t-s)^{\frac{1}{2}+\gamma_0} \sqrt{s}}  ds &\\
  &\leq C_D(m+C_K) +4C_K C_D A_K I(\frac{1}{2}+\gamma_0,\frac{1}{2}) 
 \sqrt{\tau'}.&  
  \end{align*}
Hence, since $X_{\frac{1}{2}+\gamma_0}\subset C^{1,\gamma}(\overline{B})$, we deduce that there exists $A_K'>0$ depending only on $K$ such that  
$$\|w(\tau')\|_{C^{1,\gamma}(\overline{B})} \leq \frac{A_K'}{\tau'^{(\frac{1}{2}+\gamma_0)}}=:A_K''.$$
\underline{Last step :}
Let $t'\geq \frac{t_0}{4}$. Since $\tau'\leq\frac{t_0}{4}$, we can apply the same arguments by taking $w_0(t'-\tau)$ as initial data instead of $w_0$, so 
we obtain $$\mbox{for all } t'\geq \frac{t_0}{4},\,\|w(t')\|_{C^{1,\gamma}(\overline{B})} \leq A_K'' .$$
Finally, coming back to $u(t)$, thanks to formula (\ref{def_u}), we get an upper bound $D_K$ for $ \|  u(t) \| _{C^{1,\frac{\gamma}{N}}}$ valid for any $u_0\in Y_m$ such that $\mathcal{N}[u_0]\leq K$. 
\end{proof}

\medskip

\begin{proof}[Proof of Lemma \ref{T_dynamical_system}]

 Thanks to Lemma \ref{traj_globales}, we know that $T(t)$ 
is well defined for all $t\geq 0$ and by definition of a classical solution, $T(t)$ maps $Y_m^1$ into $Y_m^1$. \\
ii) is clear by uniqueness of the global classical solution.\\
iv) comes from the fact that $u\in C((0,\infty),C^1([0,1]))$.\\
iii) Let $t> 0$, $u_0\in Y_m^1$ and $(u_n)_{n\geq 1}\in Y_m^1$.\\
 Assume that $u_n \underset{n \rightarrow \infty}{\overset{C^1}{\longrightarrow}} u_0$.
Let us show that $u_n(t) \underset{n \rightarrow \infty}{\overset{C^1}{\longrightarrow}} u(t)$.\\
We proceed in two steps.\\
\underline{First step :}  We want to show that if $u_n \underset{n \rightarrow \infty}{\overset{C^1}{\longrightarrow}} u_0$, then 
$u_n(t) \underset{n \rightarrow \infty}{\overset{C^0}{\longrightarrow}} u(t)$.

Actually, this has already been done in the proof of Theorem \ref{thm_existence_u} vii) (in the first step, second part). 
Indeed,  the argument there shows that if all
 the $u_n$ exist on a common interval $[0,T_0]$,  then we have
$$\|u_n-u\|_{\infty,
 [0,1]\times [0, T_0]} 
 \underset{n\rightarrow \infty}{\longrightarrow} 0.$$
 But here, for all n, $T_{max}(u_n)=+\infty$, so this result can be applied to 
$T_0=t$, which implies the result.

\underline{Second step : } since $u_n \underset{n \rightarrow \infty}{\overset{C^1}{\longrightarrow}} u_0$, $\|u_n\|_{C^1}$ is bounded so there exists $K>0$ such that for all $n\geq 1$, $\mathcal{N}[u_n]\leq K$.
Then, from Lemma \ref{lem_compacite}, since $t>0$, $\{u_n(t),\;n\geq 1\}$ is relatively compact in $Y_m^1$ and has a single accumulation point $u(t)$ from first step. Whence the result.

\end{proof}

\medskip

\begin{proof}[Proof of Lemma \ref{G_Lyapunov_functional}]
Let $t>0$. We can differentiate the integral by applying Lebesgue's dominated theorem.
Indeed, let $\eta>0$ small enough so that $$I=[t-\eta,t+\eta]\subset (0,T_{max}).$$ Note : here, for $0\leq m<2$, $T_{max}=+\infty$.\\
Since $u\in C(I,C^1([0,1]))$, then $\frac{u}{x}$ is bounded on 
$I\times[0,1]$.
Since moreover, by Theorem \ref{thm_existence_u} vii), for all $t\in I$, $u_x(t)>0$ on $[0,1]$, then $\ln(u_x)$ is bounded on $I\times[0,1]$.\\

Let $(t,x)\in (0,T_{max}(u_0))\times (0,1]$. We recall that 
$$u_x(t,x)=\tilde{w}\left(\frac{t}{4},\sqrt{x}\right)+\sqrt{x}\; \frac{\tilde{w}_r}{2}
\left(\frac{t}{4},\sqrt{x}\right).$$
Hence, $$u_{xx}(t,x)=\frac{3}{4\sqrt{x}}\tilde{w}_r\left(\frac{t}{4},\sqrt{x}\right)+\frac{1}{4} \tilde{w}_{rr}
\left(\frac{t}{4},\sqrt{x}\right)$$
and
$$u_{xxx}(t,x)=-\frac{3}{8x^{\frac{3}{2}}}\tilde{w}_r\left(\frac{t}{4},\sqrt{x}\right)+\left(\frac{3}{8x}+\frac{1}{4}\right)\tilde{w}_{rr}
\left(\frac{t}{4},\sqrt{x}\right)+
\frac{1}{8\sqrt{x}} \tilde{w}_{rrr}
\left(\frac{t}{4},\sqrt{x}\right).$$
Since $$u_t=x\, u_{xx}+u\, u_x$$ and $$u_{t,x}=x\,u_{xxx}+[1+u]u_{xx}+{u_x}^2,$$
it is now easy to see that $u_t$ and $u_{t,x}$ are bounded on $I\times[0,1]$. \\
Since $u\in C^2((0,T_{max})\times (0,1])$ , then $u_{t,x}=u_{x,t}$.\\
Finally, $\ln( u_x)\, u_{t,x}-\frac{u\,u_t}{x}$ is bounded on $I\times[0,1]$.  
Hence, by direct calculation,
\begin{align*}
\frac{d}{dt}\mathcal{G}[u(t)]=&\int_0^1 \ln (u_x)\, u_{t,x}-\frac{u\,u_t}{x}
=-\int_0^1 [\frac{u_{xx}}{u_x}+\frac{u}{x}] u_t\\
=&-\int_0^1 \frac{{u_t}^2}{xu_x}
\end{align*}
where  an integration by parts was made, using that $u_t(t,0)=u_t(t,1)=0$.\\
It is easy to see that $\mathcal{G}$ is continuous on $\mathcal{M}$,
$\mathcal{G}$ is nonincreasing on the trajectories, so
 we have proved that $\mathcal{G}$ is a Lyapunov function. Now, assume that $$\mathcal{G}[u(t)]=\mathcal{G}[u_0]$$ for all $t\geq 0$.
 This implies $$\int_0^t\int_0^1 \frac{{u_t}^2}{xu_x}=0$$ so $$u_t=0$$
 on $[0,1]\times [0,t]$ for all $t\geq 0$. Hence, by continuity, for all $t\geq 0$,
 $$u(t)=u_0$$  i.e. $u$ is a steady state of $(PDE_m)$.
\end{proof}

\medskip

\begin{proof}[Proof of Lemma \ref{N=2_lem_convergence_C1}]
Let $u_0\in Y_m$ and u the global solution of $(PDE_m)$ with initial condition $u_0$. Let us set $$u_1=u(1)\in Y_m^1.$$
To get the result, we just have to study $\underset{t\rightarrow +\infty}{\lim} T(t)u_1$. \\
Thanks to Lemma $\ref{lem_compacite}$, $\gamma_1(u_1)$ is relatively compact in $Y_m^1$ and 
since $\mathcal{G}$ is a strict Lyapunov functional for $(T(t))_{t\geq 0}$, we know by Lasalle's invariance principle (Proposition \ref{equilibre}) that 
the $\omega$-limit set $\omega(u_1) $ is non empty and contains only stationary 
solutions. But
since there exists only one  
steady state $U_a$ where $$a=\frac{m}{1-\frac{m}{2}}, $$ then  $$\omega(u_1)=\{U_a\}$$ so $$T(t)u_1 \underset{t 
\rightarrow + \infty}{ \longrightarrow U_a}.$$ 
\end{proof}

\textbf{Acknowledgments :} The author would like to thank Philippe Souplet for helpful comments about this paper.

 \end{document}